\documentclass{amsart}
\usepackage{amsmath,amssymb,amsthm}
\usepackage{fancybox,ascmac}
\usepackage{bbm}
\usepackage[dvips]{graphicx}
\usepackage{color}

\newtheorem{dfn}{Definition}[section]
\newtheorem{thm}[dfn]{Theorem}
\newtheorem{lem}[dfn]{Lemma}
\newtheorem{cor}[dfn]{Corollary}

\newtheorem{prop}[dfn]{Proposition}\makeatletter

\newcommand{\ca}{{\rm Cap}}
\newcommand{\Z}{{\mathbb Z}}
\newcommand{\RR}{{\mathbb R}}

\newcommand{\R}{{\mathcal R}}

\numberwithin{equation}{section}

\usepackage{comment}
\usepackage[utf8]{inputenc}

\title[Deviations of the intersection with general kernel]{Deviations of the intersection of Brownian Motions in dimension four with general kernel }
\author{Arka Adhikari}
\address{Arka Adhikari\hfill\break
Department of Mathematics, Stanford University, Stanford, CA, USA}
\thanks{Research supported by NSF grant DMS 2102842 (A.A.) and JSPS KAKENHI Grant-in-Aid  for Early-Career Scientists (No.~JP20K14329) (I.O.).}
\email{arkaa@stanford.edu}
\author{Izumi Okada}
\address{Izumi Okada\hfill\break
Department of Mathematics and Informatics, Faculty of Science, Chiba University, Chiba 263- 8522, Japan}
\email{iokada@math.s.chiba-u.ac.jp}


\begin{document}

\maketitle

\begin{abstract}
 In this paper, we find a natural four dimensional analog of the moderate deviation results of Chen \cite{Chen04} for the mutual intersection of two independent Brownian motions $B$ and $B'$. In this work, we focus on understanding the following quantity, for a specific family of kernels $H$,
\begin{equation*}
  \int_0^1 \int_0^1 H (B_s - B'_t) \text{d}t \text{d}s . 
\end{equation*}
Given  $H(z) \propto \frac{1}{|z|^{\gamma}}$ with $0 < \gamma \le 2$,  we find that the deviation statistics of the above quantity can be related to the following family of inequalities from analysis,
\begin{equation} \label{eq:maxineq}
\inf_{f: \|\nabla f\|_{L^2}<\infty} \frac{\|f\|^{(1-\gamma/4)}_{L^2} \|\nabla f\|^{\gamma/4}_{L^2}}{ [\int_{(\mathbb{R}^4)^2} f^2(x) H(x-y) f^2(y) \text{d}x \text{d}y]^{1/4}}.
\end{equation}
Furthermore, in the case that $H$ is the Green's function, the equation \eqref{eq:maxineq} will correspond to the generalized Gagliardo-Nirenberg inequality; this is used to analyze the Hartree equation in the field of partial differential equations. Thus, in this paper, we find a new and deep link between the statistics of the Brownian motion and a family of relevant inequalities in analysis.
\end{abstract}

 \tableofcontents

\section{Introduction}
\subsection{Motivation and Related Background}


{In this paper, we find a four dimensional analog of the moderate deviation results of Chen \cite{Chen04} for the mutual intersection of two independent Brownian motions; other related papers include \cite{BC04, BCR06, Chen04, Chen05, CL, KM}. 
Let $\tau_1$ and $\tau_2$ be two independent exponential random variables with rate 1 and $B,B'$ be two independent Brownian motions starting at $0$.   
Consider a kernel $H$ of the form $K *K$ with $K$ a positive function, such that 
$|H(z)| \le \frac{C}{|z|^{\gamma}}$ for some constant $C>0$ and $0 < \gamma \le 2$. 
We study the moderate deviation for the following quantity,
\begin{equation*}
\mathcal{G}_H
:=\int_{0}^{\tau_1} \int_{0}^{\tau_2} H(B_t - B'_s) \text{d}t \text{d}s. 
\end{equation*}
Then, the constant $\alpha_H$ that determines the large deviation behavior of $\mathcal{G}_H$ can be expressed as
\begin{equation*} 
2 \log \alpha_H := \lim_{n\to \infty}\frac{1}{n} \log \frac{1}{(n!)^2} \mathbb{E}[ (\mathcal{G}_H)^n].
\end{equation*}
This is related to the following optimization problem. 
For a given parameter $\theta$, denote,
\begin{equation*} 
M(\theta)
:= \sup_{ \substack{\|g\|_{L^2}=1 \\ \|\nabla g \|_{L^2}< \infty}} 
\theta
  \left[ \int_{(\mathbb{R}^4)^2} g^2(x) H(x-y) g^2(y) \text{d}x \text{d}y \right]^{1/2} - \frac{1}{2} \int_{\mathbb{R}^4} |\nabla g(x)|^2 \text{d}x. 
\end{equation*}
 $\alpha_H$ is exactly the constant for which $M(\alpha_H^{-1}) = 1$. 

When the kernel $H$ has nice scaling properties, namely, $H(z) \propto \frac{1}{|z|^{\gamma}}$ for $0 < \gamma \le 2$, 
one can check that $\alpha_H$ will be related to the optimal constant $\kappa_H$ in the following generalized Gagliardo-Nirenberg inequality (ref. \cite[Theorem 2.3]{FY15}),
\begin{equation} \label{eq:genGagnir}
 \left[ \int_{(\mathbb{R}^4)^2} g^2(x) H(x-y) g^2(y) \text{d}x \text{d}y \right]^{1/4}  \le \kappa_H \| g\|_{L^2}^{1- \gamma/4} \|\nabla g\|_{L^2}^{\gamma/4}.
\end{equation}
Indeed, in such a case, $\alpha_H = \kappa_H^2 \left(\frac{\gamma}{2}\right) \left( \frac{2\gamma}{4-\gamma}\right)^{\frac{\gamma-4}{4}}$. 
Eventually, we will obtain the following moderate deviation result: 
\begin{equation*} 
\lim_{T\to \infty} T^{-\frac{2}{\gamma}} \log P\bigg(\int_{0}^1 \int_{0}^1 H(B_t - B'_s) \text{d}t \text{d}s \ge T \bigg)
=-\frac{\gamma}{2} (\frac{4-\gamma}{4})^{\frac{4-\gamma}{\gamma}} \alpha_H^{-\frac{4}{\gamma}}
=- \kappa_H^{-\frac{8}{\gamma}}.
\end{equation*}

Furthermore, we remark that the only kernels $H$ that can satisfy an inequality of the form 
$$
 \left[ \int_{(\mathbb{R}^4)^2} g^2(x) H(x-y) g^2(y) \text{d}x \text{d}y \right]^{1/4} \le \kappa_H  \| g\|_{L^2}^{1-c} \|\nabla g\|_{L^2}^{c}
$$
must satisfy $|H(z)| = \text{O}\left( \frac{1}{|z|^{4 c}}\right)$. 
In addition, one can check directly that if $H(z) = \frac{1}{|z|^{4 -\gamma}}$, then it has a convolutional square root of the form $\frac{1}{|z|^{4- \gamma/2}}$. In this sense, we can relate the large deviation statistics of a generalized intersection of Brownian motions with the most important inequalities of the form \eqref{eq:genGagnir}.

In what follows, we will restrict to the case that our kernel is the Green's function of the Brownian motion in $d=4$, that is, $G(x) \propto |x|^{-2}$ 
since the proof is the same if $H(z) \propto \frac{1}{|z|^{\gamma}}$ with $0 < \gamma \le 2$.  In this case, we will use $\mathcal{G}$ to denote our central quantity of interest.
Before we proceed with discussing the details of the proof and consequences of our results, we give some comments on why introducing this is the natural interaction between two Brownian motions in $d=4$. 
Although one can make sense of the notion of self-intersection in $ d=4$, the answer is not interesting since two Brownian motions are with exceedingly high probability, will only intersect finitely many times before never intersecting again. Thus, the mutual intersection does not have good scaling properties in $d=4$. 

By contrast, one would expect that self-intersection moderated by the Green's function kernel in $d=4$ would have the same scaling behavior as the usual self-intersection in $d=2$. 
If we apply the self-similar scaling $B_{\cdot} \to  \frac{1}{\sqrt{t}} B_{ t \cdot }$, we see that $\int_{0}^T \int_{0}^T G(B_t - B'_s) \text{d}t \text{d}s$ has the same distribution as $T \int_{0}^1 \int_{0}^1 G(B_t - B'_s) \text{d}t \text{d}s$. 
This is exactly the same critical scaling behavior as the self-intersection in $d=2$. 

Beyond just giving us a generalization of the moderate deviation results of Chen \cite{Chen04} to $d=4$, our results and proofs reveal connections between the properties of Brownian motions and central quantities in the analysis of differential equations. In $d=2$, Chen \cite{Chen04} revealed the connection between the constant that appears in the moderate deviation analysis of the intersection and the optimal constant in an appropriate Gagliardo-Nirenberg inequality in $d=2$ i.e.,
\begin{align*}
\inf_{f} \frac{\|f\|_{L_2}^{1/2}\|\nabla f\|_{L_2}^{1/2}}{\|f\|_{L_4}}.
\end{align*}

Here, we find  a relationship between the exact constant in the moderate deviation study of $\mathcal{G}$ and what is known as the generalized Gagliardo-Nirenberg inequality in the study of partial differential equations {(ref. \cite{FY15, MS}) i.e., }
\begin{align*}
\inf_{f} \frac{\|f\|_{L_2}^{1/2}\|\nabla f\|_{L_2}^{1/2}}{\left[\int_{(\mathbb{R}^4)^2}f^2(x) |x-y|^{-2} f^2(y) \text{d}x \text{d}y\right]^{1/4} }.
\end{align*}
If we look at  \cite[Theorem 2.3]{FY15}, this inequality is derived from the Hardy-Littlewood-Sobolev inequality and is used to study the Hartree equation. 
Hence, we find a new relationship between the intersection of Brownian motions and the field of analysis.

 
The final application of our result lies in the study of the capacity of the range of a random walk. The second author was originally motivated to study this question to extend the results of the paper \cite{DO}. 
The original goal was to derive a similar result for  moderate deviations of a random walk and to further explore the link between the capacity in general dimension $d$ and the self-intersection in dimension $d-2$. As seen in \cite{DO}, the asymptotics of the capacity of a random walk is controlled by the self-intersection moderated by the Green's function. 
In a forthcoming paper, we can prove a moderate deviation principle for the capacity. 

\subsection{Strategy and Mathematical Description}

As is well-known in the study of large deviations,  the moderate deviation behavior of a positive random variable can be determined via the exact asymptotics of large moments of the random variable. For example, in our case, to understand the moderate deviation behavior of $\mathcal{G}$, one would need to compute quantities such as,
\begin{equation*}
\lim_{n \to \infty} \frac{1}{n} \log \frac{1}{(n!)^2} \mathbb{E}[ \mathcal{G}^n].
\end{equation*}

Chen \cite{Chen04} has performed such computation in the context of the intersection. 
Indeed, the main tool that he has applied is a nice expression for general $n$-th moments with regards to the  mutual-intersection, Le Gall's formula. On a formal level, 
if one considers just two Brownian motions,we have that, 
\begin{equation*}
\int_{0}^n \int_{0}^n \delta (B_t - B'_s) \text{d}t \text{d}s
= \int_{x \in \RR^d} \text{d} x \int_{0}^n \delta (B_t - x) \text{d}t \int_{0}^n \delta (B'_s - x) \text{d}s, 
\end{equation*}
where $\delta$ is the usual Kronecker delta function. 
Thus,
\begin{align*}
&\mathbb{E}\left[\left(\int_{0}^n \int_{0}^n  \delta (B_t - B'_s) \text{d}t \text{d}s \right)^m\right] \\
= & \int_{(\RR^d)^m} \text{d} x_1\ldots \text{d} x_m
 \left(\mathbb{E}\left[  \sum_{\rho} \int_{0\le t_1\le t_2 \le \ldots \le t_n}  \prod_{j=1}^m  \delta (B_{t_j} - x_{\rho(j)}) \text{d}t_j \right] \right)^2.
\end{align*}

The main benefit of this formula is that, through the introduction of the points $x_1,\ldots,x_m$, 
we see that we can separate $B$ and $B'$ and treat the expectations separately. 
Furthermore, one can explicitly compute the expectations above and write it in terms of the transition probability of the Brownian motion. Under the right setup, one can observe that the computation above resembles a Markov transition probability. Thus, after careful manipulation, one can eventually relate the quantity above to finding the eigenvalues of an appropriate symmetric operator; 
finding this maximum eigenvalue 
 can now  be readily phrased as an optimization problem over an appropriate subspace. 
For example, one now has access to formulas resembling the Feynman-Kac formula, which allow one to compute functions of the form $\mathbb{E}[\int_{0}^1 f(B_t) \text{d}t]$. 

However, none of these heuristics can possibly work if one does not separate $B$ and $B'$ from each other; if one essentially has to deal with two time parameters simultaneously, then there is no way to relate this quantity to a Markov transition probability.  When considering the intersection of two Brownian motions moderated by the Green's function, there is not an obvious way to split the two Brownian motions. 
Namely, the function $G(B_s-B'_t)  \propto  \frac{1}{|B_s - B'_t|^2}$ does not naturally lend itself to a splitting of $B$ and $B'$ and it appears that a computation of the moments fundamentally has to deal with some correlation between $B$ and $B'$. 

However, in this paper, we are able to find a means of circumventing this difficulty. 
We first express $G(B_s-B'_t)$ as $\int_{z \in \mathbb{R}^4} \tilde{G}(B_s - z) \tilde{G}(B'_t -z) \text{d}z$. Here $\tilde{G}$ is the convolutional square root of $G$; on a formal level, this allows one to separate out $B$ and $B'$ from each other in the formula.  If we perform the splitting, we get access to multiple computational tools, such as the Feynman-Kac formula. Indeed, one can obtain a lower bound on asymptotic moments relatively straighforwardly via an appropriate application of the Feynman-Kac formula. However, there are still multiple challenges to get an appropriate upper bound.

The main tool to derive an  upper bound is to approximate the moment computation by a Markov transition kernel. If one has a Markov transition kernel representation of the upper bound, then one can represent the upper bound in terms of finding the largest eigenvalue of an appropriate linear operator. However, there are multiple difficulties to deal with in order to derive a Markov transition kernel approximation. In the context of the computation in the Brownian motion, a computation of the $n$-th moment naturally expresses computations as a sum over permutation over $n$ points $x_1,\ldots,x_n$(see equation \eqref{eq:momentBrownian}). A natural Markov kernel approximation would replace the sum of configurations $y_1,\ldots,y_n$, where any configuration $y_1,\ldots,y_n$ is allowed to be a permutation of $x_1,\ldots,x_n$ over configurations $y_1,\ldots,y_n$ where each of the $y_i$ is allowed to be one of the $x_1,\ldots,x_n$ independently of each other.  However, such an approximation will lose exponential factors unless that size of $|\{x_1,\ldots,x_n\}|$ is  far less than $n$; this is only possible if the total state space is finite. 
To make this justification rigorous, we had to appropriately discretize $\mathbb{R}^4$ and argue that there was little loss in making such manipulations. 
In addition, such a justification involved regularizing the singularity of $G$ near the origin. 

Furthermore, the natural Markov transition kernel representation that can be derived from the moments is a rather cumbersome expression involving multiple functions with awkward normalization conditions phrased in terms of the convolutional square root $\tilde{G}$ (see equation \eqref{eq:optrho}). In order to relate this upper bound to the lower bound, one must find a way to transform the optimization problem equation \eqref{eq:optrho} to the constant coming from the modified Gagliardo-Nirenberg inequality. 
The particular form of the modified Gagliardo-Nirenberg inequality is not merely incidental to the proof, it was a necessity in order to bridge the different ways of obtaining the lower and upper bounds for the asymptotic moments.


\subsection{Main Results}

Let $G(z) = \int_0^{\infty} p_t(z) \text{d}t$ and $p_t(z)$ is the transition density that a Brownian motion will reach point $z$ at time $t$. 
 Our main result proves a moderate deviation principle for $\mathcal{G}([0,1]):=\int_{0}^1 \int_0^1 G(B_t - B_s') \text{d}t \text{d}s$. 

As we have stated, the Green's function for the Brownian motion is an inverse polynomial; 
we can relate $\mathcal{G}([0,T])=_d T \mathcal{G}([0,1])$. 
Thus, once we obtain a large deviation principle for $\mathcal{G}([0,1])$, we can extend it to that of $\mathcal{G}([0,T])$. 
 We find that  $\mathcal{G}([0,1])$ is related to best constant of the modified Gagliardo-Nirenberg inequality.  
 Namely, it is the smallest constant $\tilde{\kappa}(4,2)$ such that the following inequality should hold: 
\begin{align}\label{GNconst}
\left[ \int_{(\mathbb{R}^4)^2} g^2(x) G(x-y)g^2(y) \text{d}x \text{d}y \right]^{1/4} \le \tilde{\kappa}(4,2) \|g\|^{1/2}_{L^2} \|\nabla g\|^{1/2}_{L^2}.
\end{align}

\begin{thm}\label{p3}
We have the following large deviation behavior on $\mathcal{G}([0,1])$. 
For $d=4$, 
\begin{align*}
\lim_{t\to \infty}
\frac{1}{t} \log P(  {\mathcal{G}([0,1])}  \ge t)=-\tilde{\kappa}^{-4}(4,2) . 
\end{align*}
\end{thm}

Now, we also claim the result for general kernel. 
\begin{cor}\label{corgen}
Assume $H(z) \propto \frac{1}{|z|^{\gamma}}$ with $0 < \gamma \le 2$. 
For $d=4$, 
\begin{equation*} 
\lim_{T\to \infty} T^{-\frac{2}{\gamma}} \log P\bigg(\int_{0}^1 \int_{0}^1 H(B_t - B'_s) \text{d}t \text{d}s \ge T \bigg)
=- \kappa_H^{-\frac{8}{\gamma}}, 
\end{equation*}
where $\kappa_H$ is the optimal constant in \eqref{eq:genGagnir}. 
\end{cor}

Next, we also consider the following self-intersection local time of the Brownian motion in $d=4$ moderated by the Green's function: 
\begin{equation*}
\beta_t :=
\int_{0}^t \int_{0}^s G(B_l - B_s) \text{d}l \text{d}s
-E \bigg[ \int_{0}^t \int_{0}^s G(B_l - B_s) \text{d}l \text{d}s \bigg]. 
\end{equation*}
Note that $\int_{0}^t \int_{0}^s G(B_l - B_s) \text{d}l \text{d}s$ does not exist but as in \cite{As5}, we can define $\beta_t$ by renormalization. 
By \cite{As5, DO}, we find that: 
there exists a value $\gamma_\beta$ such that
\begin{align*}
Ee^{\gamma \beta_1}
\begin{cases}
<\infty & \text{ if } \gamma<\gamma_\beta,\\
=\infty & \text{ if } \gamma>\gamma_\beta.
\end{cases}
\end{align*}

The ordinary self-intersection local time of the Brownian motion in $d=2$ was estimated in \cite{BC04}. 
From our mdoerate deviation results on $\mathcal{G}$, we can show moderate deviation resutls on $\beta_t$; these correspond to \cite[Theorems 1.1 and 1.2]{BC04}. 
We can also obtain results corresponding to \cite[Theorems 1.3-1.5]{BC04} as a corollary by using very similar methods; as such, we omit the proof. 
\begin{thm}\label{p4}
We have
\begin{align}\label{p4*}
\lim_{t\to \infty}
\frac{1}{t} \log P(  \beta_1  \ge t)=-\tilde{\kappa}^{-4}(4,2) . 
\end{align}
In particular,  $\gamma_\beta
=\tilde{\kappa}^{-4}(4,2)$. 
\end{thm}

Finally, we introduce the resulst of our forthcoming paper regarding the moderate deviation of the capacity of a simple random walk, which is one of the motivation of this paper. 
We have the following moderate deviation behavior for $\mathfrak{G}_n:=\sum_{i=1}^n \sum_{l=1}^n G_d(S_i-\tilde{S}_l)$, 
where $S$ and $\tilde{S}$ are independent simple random walks on $\Z^4$ and $G_d$ is the discrete Green's function. 
Let $b_n= o(n)$ and $\lim_{n \to \infty} b_n = \infty$. 
Then, we have, for $\lambda>0$, 
\begin{align*}
&\lim_{n\to \infty}
\frac{1}{b_n} \log P( \mathfrak{G}_n \ge  \lambda n b_n)=-\tilde{\kappa}^{-4}(4,2) \lambda.
\end{align*}

As we have mentioned earlier in the introduction, we can use our main result to obtain a moderate deviation principle for the capacity of a range of a random walk. 
Let $\tau_A$ denote the first positive hitting time of a finite set $A$ by a simple random walk $S$. 
Define 
\begin{align*}
\ca (A)
:= \sum_{x \in A} 
P^x(\tau_A =\infty)
\end{align*}
and $\R[a,b] :=  \{S_a,\ldots,S_b\}$. If $a=0$, we simply write it as $\R_b$.  
As observed in the papers \cite{As5,DO}, the capacity of the range of the random walk can be carefully decomposed as the sum of the capacities of the first and second halves of the random walk as well as a term representing the `mutual capacity' of interaction between the first and second halves.
Namely, one has that,
\begin{equation*} 
   \ca( \R_{2n}) = \ca(\R_n) + \ca(\R[n+1,2n])
     -  \chi_c(\{S_i\}_{i=0}^n, \{S_i\}_{i=n+1}^{2n}) 
\end{equation*}
for some function $\chi_c$. 
Once one does this, one will observe that the main contribution to the large deviation behavior will come from the terms $\chi_c$. 
As investigated in \cite{DO}, the term $\chi_c$ can be marginally simplified to be of the form of $\frac{\pi^2}{64(\log n)^2} \mathfrak{G}_n$. 
Thus, we estimate the following : 
for some  $b_n \to \infty$, 
\begin{align*}
\lim_{n\to \infty}
\frac{1}{b_n} \log 
P\bigg( \ca (\R_n)-E \ca (\R_n) \le -\frac{\lambda n}{(\log n)^2}b_n \bigg). 
\end{align*}

Finally, we explain the contents of this paper. 
The central Theorem \ref{p3}  is divided into Sections \ref{sec:contlargedeviation} and \ref{sec:contupperbound}. 
This relates the large deviations of Theorem \ref{p3} to an optimization problem defined by the constant $\rho$ as in equation \eqref{eq:optrho}. 
To relate this constant to a more fitting form, 
we have an intermediate Section \ref{sec:rhoGagnir} which relates the quantity $\rho$ to the modified Gagliardo-Nirenberg inequality. 
In Section \ref{sec:self-inter}, we estimate the self-intersection of the Brownian motion moderated by the Green's function, which corresponds to the proof of  Theorem \ref{p4}. 
Appendix \ref{ref:appendixA} contains estimates that regularize the singularity of the Green's function around the $0$.  At the beginning of Section \ref{sec:contupperbound}, we split $G$ into a component supported near the origin and another away from the origin; the results of Appendix \ref{ref:appendixA} show that the component  supported near the origin does not contribute asymptotically to the large deviation statistics. Appendix \ref{sec:AppendixB} allows us to analyze the modified Kernels obtained via the discretization and compactification procedure in Section \ref{sec:contupperbound}; in particular, it appears in the proof of Lemma \ref{lem:removeepsilon} to remove the effects of discretization.



\section{Large Deviation for the intersection of Brownian Motions: The Proof of Theorem \ref{p3} } \label{sec:contlargedeviation}


In this section, we will consider the large deviation of the intersection moderated by the Green’s function kernel of the Brownian motion. 
Recall that our basic quantity $\mathcal{G}$ is given by
\begin{equation*}
\mathcal{G}:=\int_{0}^{\tau_1} \int_{0}^{\tau_2} G(B_t-B'_s) \text{d}t \text{d}s.
\end{equation*}
Here, as before, $\tau_1$, $\tau_2$ are exponential random variables. 
This is in contrast to $\mathcal{G}([0,1])$, in which both Brownian motions vary from time $0$ to $1$. 
We remark that $\mathcal{G}([0,1])$ has the following scaling property,
\begin{equation} \label{eq:scaling}
   \mathcal{G}([0,t])=_d t \mathcal{G}([0,1]).
\end{equation}

In the following result, which is the main theorem in this section, we compute the moment. 
Our strategy will be to write $G$ in terms of its convolutional square root $G= \tilde{G}* \tilde{G}$. 
One can directly compute the convolutional square root as $\tilde{G}(x) = \int_0^\infty \frac{1}{\sqrt{\pi t}} p_t(x) \text{d}t$. 
Thus, we see that $\tilde{G}$ is positive an has the asymptotics, $\tilde{G}(x) \propto \frac{1}{|x|^3}$. 
Also, we use $P_{\tau}$ to denote the probability density that a Brownian motion killed by an exponential variable with rate 1 reached the point $x$ at some time. 
Namely, $P_{\tau}(x) = \int_{0}^{\infty} e^{-t}p_t(x) \text{d}t$.  
\begin{thm} \label{thm:onGBtau}
Consider $\mathcal{G}$ as defined earlier. We have the following expression for the large moments:
\begin{equation*}
\lim_{n\to \infty}
\frac{1}{n}\log \frac{1}{(n!)^2} \mathbb{E}[\mathcal{G}^n] = 2 \log \rho.
\end{equation*}

Here, $\rho$ is the solution to the following optimization problem, that is, 
\begin{equation} \label{eq:optrho}
    \rho: =
\sup_{\substack{f \in L^2_G\\ k:  \int_{\mathbb{R}^4} k^2(z) \text{d}z =1}} 
   \int_{(\mathbb{R}^4)^4} f(\tilde{z},\tilde{e}) \sqrt{k}(\tilde{z}) \tilde{G}(\tilde{e})P_{\tau}(\tilde{z}+\tilde{e}-z-e) \tilde{G}(e) \sqrt{k}(z) f(z,e) \text{d}\tilde{z} \text{d} \tilde{e} \text{d}z \text{d}e,
\end{equation}
and $f \in L^2_G$ is the space of functions that satisfies $\int_{(\mathbb{R}^4)^2}f^2(z,e) \tilde{G}(e)\text{d}z \text{d}e =1$.
\end{thm}

Note that we see that we can write $\mathcal{G}$ as, 
\begin{equation*} 
\mathcal{G}
    := \int_{\mathbb{R}^4} \text{d} z\int_{0}^{\tau_1}  \tilde{G}(B_t-z) \text{d}t \int_{0}^{\tau_2} \tilde{G}(B'_s-z)\text{d}s.
\end{equation*}
At this point, we can try to take powers of the following expression and compute the resulting moments: 
\begin{align}\label{eq:momentBrownian}
    \mathbb{E}[\mathcal{G}^n] 
    =  \int_{(\mathbb{R}^4)^n} \text{d}z_1\ldots \text{d}z_n\left[ \sum_{\rho} \int_{(\mathbb{R}^4)^n} \prod_{i=1}^n \tilde{G}(x_i-z_{\rho(i)}) P_{\tau}(x_i - x_{i-1}) \text{d}x_1\ldots\text{d}x_n \right]^2. 
\end{align}

Analysing this expression carefully will allow one to deduce Theorem \ref{thm:onGBtau}. 
By scaling, we can almost relate this to the more standard expression $\mathcal{G}$. 
One issue here is that in order to apply the scaling argument, one needs the times $\tau_1$ and $\tau_2$ to match. 
This clearly cannot be true for random, independent $\tau_1$ and $\tau_2$.
 However, we have inequalities to relate the expressions $\mathcal{G}$ with $\mathcal{G}([0,1])$. 
If we consider the general expression $\mathcal{G}_{t_1,t_2}:= \int_{ \mathbb{R}^4} \text{d}z \int_{0}^{t_1} \tilde{G}(B_t-z) \text{d}t \int_0^{t_2} \tilde{G}(B'_s-z) \text{d}s$, we have the following analog of Le Gall's moment formula,
\begin{equation*}
\begin{aligned}
&\mathbb{E}[(\mathcal{G}_{t_1,t_2})^n]\\
& = \int_{(\mathbb{R}^4)^n} \text{d}z_1 \ldots \text{d}z_n\\
& \times \sum_{\rho_x} \int_{(\mathbb{R}^4)^n} \text{d}x_1\ldots\text{d}x_n\int_{[0,t_1]^n} \text{d}s_1\ldots \text{d}s_n \prod_{i=1}^n  \tilde{G}(x_i-z_{\rho_x(i)}) p_{s_i - s_{i-1}}(x_i - x_{i-1})\\
&\times \sum_{\rho_y}\int_{(\mathbb{R}^4)^n}\text{d}y_1\ldots\text{d}y_n\int_{[0,t_2]^n} \text{d}r_1\ldots \text{d}r_n \prod_{i=1}^n  \tilde{G}(y_i-z_{\rho_y(i)}) p_{r_i - r_{i-1}}(y_i - y_{i-1}).
\end{aligned}
\end{equation*}
By the Cauchy-Schwartz inequality, we can relate the moments over different times $t_1 \ne t_2$ to moments using the same time.
Noting that $\tilde{G}$ is a positive quantity, we have that 
\begin{equation} \label{eq:Cauch-Schwmom}
\mathbb{E}[(\mathcal{G}_{t_1,t_2})^n] \le \mathbb{E}[(\mathcal{G}_{t_1,t_1})^n]^{1/2} \mathbb{E}[(\mathcal{G}_{t_2,t_2})^n]^{1/2} = t_1^{n/2} t_2^{n/2} \mathbb{E}[\mathcal{G}([0,1])^{n}].
\end{equation}
In the lower bound direction, it is clear that,
additionally, it is clear that,
\begin{equation} \label{eq:scallowe4r4bound}
\mathbb{E}[\mathcal{G}^n] \ge \mathbb{E}[(\mathcal{G}_{\min(\tau_1,\tau_2), \min(\tau_1,\tau_2)})^n] 
= \mathbb{E}[(\min(\tau_1,\tau_2))^n] \mathbb{E}[\mathcal{G} ([0,1])^n].
\end{equation}

Combining manipulations on the exponential function along with equations \eqref{eq:scallowe4r4bound} and \eqref{eq:Cauch-Schwmom} allow one to relate the moments of $\mathcal{G}$ with those of $\mathcal{G}([0,1])$.
With equation \eqref{eq:Cauch-Schwmom} and \eqref{eq:scaling} in hand, one can perform standard manipulations on exponential functions to obtain the following limiting result on moments of $\mathcal{G}([0,1])$. 

\begin{cor} \label{cor:ongb} 
Consider the quantity $\mathcal{G}([0,1])$. We have the following moment estimates on $\mathcal{G}([0,1])$: 
\begin{equation}
    \lim_{n \to \infty} \frac{1}{n} \log \frac{1}{n!} \mathbb{E}[\mathcal{G}([0,1])^n] = 2 \log \rho + \log 2,
\end{equation}
and $\rho$, again, is the optimization  problem from equation \eqref{eq:optrho}.
\end{cor}
\begin{proof}
Given the scaling property in equation \eqref{eq:scaling}, 
the Cauchy-Schwartz inequality for the moment given in equation \eqref{eq:Cauch-Schwmom}, and the lower bound in \eqref{eq:scallowe4r4bound}, this follows from the computation on the moments of exponential random variables in the proof of \cite[Theorem 3.3.2]{Chenbook}.
    
\end{proof}

The optimization problem $\rho$ may not seem recognizable in this form, but in Proposition \ref{prop:GagNir}, we show that $\rho$ is the same as $\frac{\tilde{\kappa}(4,2)^2}{\sqrt{2}}$, where $\tilde{\kappa}(4,2)$ is the optimal constant in the modified Gagliardo-Nirenberg inequality \eqref{GNconst}. 
Proposition \ref{prop:GagNir} shows that $\rho = \frac{\tilde{\kappa}^2(4,2)}{2}$. Using this information on the constant $\rho$ along with standard large deviation estimates derived from moment estimates on positive quantities, we can derive the proof of the main theorem.

\begin{proof}[Proof of Theorem \ref{p3}]
Once you substitute the expression $\rho = \frac{\tilde{\kappa}^2(4,2)}{\sqrt{2}}$ from Proposition \ref{prop:GagNir} to the moment estimates in Corollary \ref{cor:ongb}, this follows from \cite[Theorem 1.2.8]{Chenbook}.

\end{proof}

\begin{proof}[Proof of Corollary \ref{corgen}] 
By the same proof as that of Green's function, we can obtain the corresponding result to Proposition \ref{prop:GagNir} and Theorem \ref{thm:onGBtau}. 
Note that 
\begin{align*}
\int_{0}^T \int_{0}^T H(B_t - B'_s) \text{d}t \text{d}s
=_d T^{\frac{4-\gamma}{2}}\int_{0}^1 \int_{0}^1 H(B_t - B'_s) \text{d}t \text{d}s. 
\end{align*}
Then, if we repeat the proof of  \cite[Theorem 3.3.2 and (2.2.20)]{Chenbook}, we have
\begin{align*} 
&\lim_{n\to \infty}\frac{1}{n} \log \frac{1}{(n!)^{(\gamma/2)}} \mathbb{E}[(\int_0^1 \int_0^1 H(B_t - B_s') \text{d}t \text{d}s)^n] \\
=&2 \log \alpha_H +(2-\gamma/2) \log 2 - (2- \gamma/2) \log (2- \gamma/2)
\end{align*}
and hence we have, again, from \cite[Theorem 1.2.8]{Chenbook},
\begin{align*} 
\lim_{T\to \infty} T^{-\frac{2}{\gamma}}\log P\bigg(\int_{0}^1 \int_{0}^1 H(B_t - B'_s) \text{d}t \text{d}s \ge T \bigg)
=-\frac{\gamma}{2} (\frac{4-\gamma}{4})^{\frac{4-\gamma}{\gamma}} \alpha_H^{-\frac{4}{\gamma}}.
\end{align*}
Since $\alpha_H = \kappa_H^2 \left(\frac{\gamma}{2}\right) \left( \frac{2\gamma}{4-\gamma}\right)^{\frac{\gamma-4}{4}}$, we obtain the desired result. 
\end{proof}

\subsection{Lower Bound for the intersection of Brownian Motions} \label{sec:lwrbndbrown}

We will prove Theorem \ref{thm:onGBtau} by proving corresponding upper and lower bounds. 
In this subsection, we prove the following lower bound estimate on the moments of $\mathcal{G}$.
\begin{thm} \label{thm:onGBtaulwr} 
Consider $\mathcal{G}$.  
We have the following  lower bound for the large moments: 
\begin{equation*}
  \liminf_{n \to \infty}
\frac{1}{n}\log \frac{1}{(n!)^2} \mathbb{E}[\mathcal{G}^n] \ge 2 \log \rho.
\end{equation*}

 Here, $\rho$  is the optimization problem defined in \eqref{eq:optrho}.

\end{thm}

\begin{proof}
Recall our moment expression \eqref{eq:momentBrownian}. 
One fact of the convolutional square root $\tilde{G}(x_i,z_{\rho_x(i)})$ is that the value only depends on the difference $e_{\rho(i)}:= x_i - z_{\rho_x(i)}$. Rewriting the expression in terms of these variables gives us that the moment is given by, 
\begin{align*}
    &\mathbb{E}[\mathcal{G}^n] \\
    =&  \int_{(\mathbb{R}^4)^n} \text{d}z_1\ldots \text{d}z_n \left[\sum_{\rho} \int_{(\mathbb{R}^4)^n} \prod_{i=1}^n \tilde{G}(e_{\rho(i)}) P_{\tau}(z_{\rho(i)} +e_{\rho(i)} - z_{\rho(i-1)} - e_{\rho(i-1)}) \text{d}e_1 \ldots \text{d}e_n\right]^2.
\end{align*}
In the above expression, one should consider $\tilde{G}(e)$ as a measure on the set of $e$ variables. From direct computation, one can see that $\tilde{G}$ is a non-negative function and can function as a measure. This is key to the strategy.

Now, we let $k(z)$ be an $L^2$ function on $z$. 
Namely, $\int k^2(z) \text{d}z = 1$. 
Then, by applying the Cauchy-Schwartz inequality, we see that, 
\begin{equation} \label{eq:G_Blowerbnd}
\begin{aligned}
    &\sqrt{\frac{1}{(n!)^2}\mathbb{E}[\mathcal{G}^n]}
     = \sqrt{ \frac{1}{(n!)^2}\mathbb{E}[\mathcal{G}^n] \prod_{i=1}^n \int_{\mathbb{R}^4} k^2(z_i) \text{d} z_i}
     \\& \ge \frac{1}{n!} \int_{(\mathbb{R}^4)^{2n}} \sum_{\rho} \prod_{i=1}^n k(z_i) \text{d} z_i \prod_{i=1}^n \tilde{G}(e_{\rho(i)}) P_{\tau}(z_{\rho(i)} +e_{\rho(i)} - z_{\rho(i-1)} - e_{\rho(i-1)}) \text{d} e_{1} \ldots \text{d}e_n\\
    &=  \int_{(\mathbb{R}^4)^{2n}}\sqrt{k(z_n)} \prod_{i=2}^{n} \sqrt{k(z_i)} P_{\tau}(z_i + e_i - z_{i-1}-e_{i-1}) \sqrt{k(z_{i-1}) } \tilde{G}(e_{i})\\
    & \hspace{3 cm} \times \sqrt{k}(z_1) P_{\tau}(z_1 + e_1) \tilde{G}(e_1) \text{d}e_1 \text{d}z_1\ldots \text{d}e_n \text{d}z_n.
\end{aligned}
\end{equation}
All the terms that appear above are positive. Thus, we can restrict $\tilde{G}$ to its portion of its support and still derive a lower bound.
Let $\tilde{G}_{R,0}(z)$ denote the restriction of $\tilde{G}(z)$ to a portion of its support to when $|z| \le R$.
Furthermore, we also assume that $\sqrt{k}$ has a finite support $S$. 
These are all technical assumptions that we will remove later.

To complete our lower bound, we also need to introduce a new quantity: 
\begin{equation*}
\begin{aligned}
    &\delta:= \min_{x \in S+R} P_{\tau}(x).
\end{aligned}
\end{equation*}
With this quantity in hand, a lower bound on the last line of \eqref{eq:G_Blowerbnd} will be
\begin{equation} \label{eq:GBtransform}
    \begin{aligned}
    & \delta \int_{(\mathbb{R}^4)^{2n}} \sqrt{k(z_n)} \prod_{i=2}^{n} \sqrt{k(z_i)} P_{\tau}(z_i + e_i - z_{i-1}-e_{i-1}) \sqrt{k(z_{i-1}) } \tilde{G}_{R,0}(e_{i})\\
    & \hspace{3 cm} \times \sqrt{k}(z_1)  \text{d}e_1 \text{d}z_1\ldots \text{d}e_n \text{d}z_n.
    \end{aligned}
\end{equation}

Now, we consider the following space of functions with corresponding inner product:
\begin{equation*}
\begin{aligned}
    &L^2_{\tilde{G},R}:= \left\{f: \int_{(\mathbb{R}^4)^2} f^2(z,e) \tilde{G}_{R,0}(e) \text{d}z\text{d}e =1 \right\},\\
    &\langle f_1,f_2 \rangle = \int_{(\mathbb{R}^4)^2} f_1(z,e)f_2(z,e) \tilde{G}_{R,0}(e) \text{d}z \text{d}e. 
\end{aligned}
\end{equation*}
We also define the following operator on this space,
\begin{equation*}
    T_{k}(f)(\tilde{z},\tilde{e}):= \sqrt{k}(\tilde{z}) \int_{(\mathbb{R}^4)^2} P_{\tau}(\tilde{z}+\tilde{e}-z-e) \sqrt{k}(z) \tilde{G}_{R,0}(e) f(z,e) \text{d}z \text{d}e.
\end{equation*}
We see that $T_k$ is a symmetric operator on our space $L^2_{\tilde{G},0}$. Namely, we have,
\begin{equation} \label{eq:symmop}
\begin{aligned}
\langle f_1, T_k f_2 \rangle = \int_{(\mathbb{R}^4)^4} f_1(\tilde{z},\tilde{e}) \tilde{G}_{R,0}(\tilde{e}) \sqrt{k}(\tilde{z}) P_{\tau}(\tilde{z}+\tilde{e}-z-e) \sqrt{k}(z) \tilde{G}_{R,0}(e) f_2(z,e) \text{d}z \text{d}e \text{d}\tilde{z} \text{d} \tilde{e}.
\end{aligned}
\end{equation}
Note that we have introduced the operator $T_k$ we can rewrite the last line of \eqref{eq:GBtransform} as,
\begin{equation} \label{eq:b4eigen}
    \delta \langle \sqrt{k}, T_k^{n-1} \sqrt{k} \rangle.
\end{equation}
Let $h_{max}(z,e)$ be the eigenfunction corresponding to the largest eigenvalue of $T_k$.  Let $h(z,e)$ be an approximator of $h_{max}(z,e)$ with the further property that
it has a lower bound $ > 0$ on the support of $\sqrt{k}$. From the form of \eqref{eq:symmop}, we see that $h(z,e)$  has no need to have support outside of the support of $\text{supp}(k) \times B_R$. $B_R$ being the ball of radius $R$ around $0$. 
Also, let us define a new quantity as follows,
\begin{equation*}
    \epsilon:= \min_{(z,e) \in \text{supp}(k) \times B_R}\frac{h(z,e)}{\sqrt{k}(z)}.
\end{equation*}
Note that $\epsilon$ exists due to our assumption that $h$ has a lower bound greater than $0$ on the set above and, furthermore, the support of $h$ cannot be outside  $\text{supp}(k) \times B_R$. 
Thus, we can thus replace \eqref{eq:b4eigen} with the lower bound,
\begin{equation*}
    \delta \epsilon^2 \langle h, T_{k}^{n-1} h \rangle \ge \delta \epsilon^2 \langle h, h_{max} \rangle^2 \langle h_{max},T_k^{n-1} h_{max} \rangle,
\end{equation*}
when $n$ is odd.
We can derive a similar lower bound when $n$ is even.
Thus, we see that,
\begin{equation*}
\begin{aligned}
    &\frac{1}{n}\log \sqrt{ \frac{1}{(n!)^2} \mathbb{E}[\mathcal{G}^n]} \ge \frac{1}{n}\log(\delta \epsilon^2 \langle h,h_{max} \rangle^2) +\\
    & \log \sup_{f \in L^2_{\tilde{G},R}} \int_{(\mathbb{R}^4)^4} f(\tilde{z},\tilde{e}) \tilde{G}_{R,0}(\tilde{e}) \sqrt{k}(\tilde{z}) P_{\tau}(\tilde{z}+\tilde{e}-z-e) \sqrt{k}(z)\tilde{G}_{R,0}(e) f(z,e)\text{d}\tilde{z}\text{d}\tilde{e}\text{d}z \text{d}e.
\end{aligned}
\end{equation*}

Now, as one considers the limit $n \to \infty$, the term $\frac{1}{n} \log(\delta \epsilon^2 \langle h, h_{max} \rangle)$ makes no contribution. 
Thus,
\begin{equation*}
\begin{aligned}
    &\liminf_{n \to \infty} \frac{1}{n} \log \frac{1}{(n!)^2} \mathbb{E}[\mathcal{G}^n] \\ & \ge 2 \log \sup_{f\in L^{2}_{\tilde{G},R}} \int_{(\mathbb{R}^4)^4} f(\tilde{z},\tilde{e}) \tilde{G}_{R,0}(\tilde{e}) \sqrt{k}(\tilde{z}) P_{\tau}(\tilde{z}+\tilde{e}-z-e) \sqrt{k}(z) \tilde{G}_{R,0}(e) f(z,e) \text{d}\tilde{z} \text{d}\tilde{e} \text{d}z\text{d}e.
\end{aligned}
\end{equation*}

Next, we observe that if a function is in $L^2_{\tilde{G},R}$ then it is in $L^2_{\tilde{G},\tilde{R}}$ for any $\tilde{R} \ge R$. 
Thus, we may first replace the restricted maximum with $\tilde{G}_{R,0}$ with,
\begin{equation*}
    \sup_{f \in L^{2}_G} \int_{(\mathbb{R}^4)^4} f(\tilde{z},\tilde{e}) \tilde{G}(\tilde{e}) \sqrt{k}(\tilde{z}) P_{\tau}(\tilde{z} + \tilde{e} -z-e) \sqrt{k}(z) \tilde{G}(e) f(z,e)
    \text{d}\tilde{z} \text{d}\tilde{e} \text{d}z\text{d}e,
\end{equation*}
where $L^2_{\tilde{G}}$ is the following space:
\begin{equation*}
    L^{2}_G:=\{f:\int_{(\mathbb{R}^4)^2} f^2(z,e) \tilde{G}(e) \text{d}z \text{d}e =1 \}.
\end{equation*}

Finally, since the choice of $k$ was arbitrary, we may finally consider the maximum over all $k$. 
Thus, we ultimately derive,
\begin{equation*}
\begin{aligned}
    &\liminf_{n \to \infty}\frac{1}{n}  \log \frac{1}{(n!)^2} \mathbb{E}[\mathcal{G}^n]
    \\ &\ge 2 \log \sup_{\substack{f \in L^2_G \\k: \int_{\mathbb{R}^4} k^2(z) \text{d}z =1}} \int_{(\mathbb{R}^4)^4} f(\tilde{z},\tilde{e}) \sqrt{k}(\tilde{z}) \tilde{G}(\tilde{e})P_{\tau}(\tilde{z}+\tilde{e}-z-e) \tilde{G}(e) \sqrt{k}(z) f(z,e) \text{d}\tilde{z} \text{d} \tilde{e} \text{d}z \text{d}e
    \end{aligned}
\end{equation*}
and we obtain the desired result. 
\end{proof}

\section{Upper Bound for the intersection of Brownian Motions} \label{sec:contupperbound}

In this section, we will establish the following result, which gives the corresponding upper bounds of the moments of $\mathcal{G}$. The following theorem, combined with Theorem \ref{thm:onGBtaulwr}, will give us Theorem \ref{thm:onGBtau}.

\begin{thm} \label{thm:uprGBtau}
Consider $\mathcal{G}$.  We have the following  upper bound for the large moments:
\begin{equation}
\limsup_{n \to \infty}\frac{1}{n}\log \frac{1}{(n!)^2} \mathbb{E}[\mathcal{G}^n] \le 2 \log \rho,
\end{equation}
 where $\rho$  is the optimization problem defined in \eqref{eq:optrho}.
   
\end{thm}

\begin{proof}

The derivation of the upper bound is far more technical. The singularity of $G$ near the origin is an obstacle; it prevents one from bounding $G$ from above by a constant in appropriate locations. However, to the scale that we are concerned, the origin has a vanishingly small contribution to the asymptotic moments. Similarly, there are some issues due to the infinite support of $G$. We first split $G$ in a main term away from the origin and $\infty$ and an error term around the origin and $\infty$. 

We first define the function $\tilde{G}_{R,\delta}(z)$ as $\tilde{G}(z)$ when $\delta \le |z| \le R$. The value will be $0$ when $|z| \ge R$. Finally, $\tilde{G}_{R,\delta}(z) = \tilde{G}_{R,\delta}(\delta)$ when $|z| \le \delta$. Once we have introduced these cutoffs, we observe the following,
\begin{equation*}
G(x-y) = \tilde{G}_{R,\delta} * \tilde{G}_{R,\delta}(x-y) + G^{\circ}(x-y).
\end{equation*}
The function $G^{\circ}(x)$ can be bounded by $G(x) \mathbbm{1}[|x| \le \delta] + [f(\delta)+ g(R)]$, where $f(\delta)$ and $g(R)$ are some functions that go to $0$ as $\delta$ goes to $0$ and $R$ goes to $\infty$ respectively.

Furthermore, we remark that for general random variables $F$ and $H$ that,
\begin{equation} \label{eq:gencompar}
\frac{1}{n} \log \frac{1}{(n!)^2} \mathbb{E}[(F+H)^n] \le \frac{1}{n} \log \left[\left(\frac{1}{(n!)^2} \mathbb{E}[F^n] \right)^{1/n} + \left(\frac{1}{(n!)^2} \mathbb{E}[H^n] \right)^{1/n}
 \right]^n.
\end{equation}
If $\rho_F$ is the limit $\frac{1}{n} \log \left[ \frac{1}{(n!)^2}\mathbb{E}[F^n] \right]$, then we see that $\rho_{F+H} \le \log (\exp[\rho_F] + \exp[\rho_H])$.

Now, it is clear that
$$
\int_{0}^{\tau_1} \int_{0}^{\tau_2} [f(\delta) + g(R)] \text{d}s \text{d} t \le \tau_1 \tau_2 [f(\delta) + g(R)].
$$
Thus, we see that,
$$
\lim_{\delta \to 0} \lim_{R \to \infty}\frac{1}{n} \log \frac{1}{(n!)^2} \mathbb{E}\left[\left(\int_{0}^{\tau_1} \int_{0}^{\tau_2} [f(\delta) + g(R)] \text{d}s \text{d} t \right)^n \right] = - \infty.
$$
From the results in the Appendix, we have from Lemma \ref{lem:Brownianorigin} that,
\begin{equation}
\lim_{\delta \to 0} \frac{1}{n} \log \frac{1}{(n!)^2} \mathbb{E}\left[ \left( \int_{0}^{\tau_1} \int_{0}^{\tau_2} G(B_t - B_s') \mathbbm{1}[|B_t - B_s'| \le \delta] \text{d}t \text{d}s \right)\right] = -\infty.
\end{equation}
Hence, we can use these facts as well as \eqref{eq:gencompar} to assert that 
$$\lim_{\delta \to 0} \lim_{R \to \infty} \lim_{n \to \infty}\frac{1}{n} \log \frac{1}{(n!)^2} \mathbb{E}[(G^{\circ}(x-y))^n] = 0.$$
Provided now that one can show the following lemma, we will be done.
\begin{lem} \label{lem:uprbnd}
It holds that, 
\begin{equation}
\begin{aligned}
&\limsup_{n \to \infty} \frac{1}{n} \log \frac{1}{(n!)^2}  \mathbb{E}\left[ \left(\int_{0}^{\tau_1} \int_{0}^{\tau_2} (\tilde{G}_{R,\delta}* \tilde{G}_{R,\delta})(B_t - B_s') \text{d}t \text{d}s \right)^n\right] \le 2 \log \rho.
\end{aligned}
\end{equation}
We denote the quantity inside the expectation on the first line as $\mathcal{G}_{R,\delta}$.

\end{lem}
\end{proof}

\subsection{The proof of Lemma \ref{lem:uprbnd}}
In this section, we will prove the following intermediary result.
\begin{lem} \label{lem:uprBRdeltaepsi}
Recall the notation $\mathcal{G}_{R,\delta}$ from Lemma \ref{lem:uprbnd}. For any choice of $M$ and $\epsilon$, we have that
\begin{equation*}
 \limsup_{n \to \infty} \frac{1}{n} \log \frac{1}{(n!)^2} \mathbb{E}[ (\mathcal{G}_{R,\delta})^n] \le 2 \log \rho_{M,R,\delta,\epsilon},
\end{equation*}
where $\rho_{M,R,\delta,\epsilon}$ is given by the following optimization problem: 
\begin{equation*}
    \sup_{\substack{ \sum_q k^2(q) = 1\\
 \sum_q \int_{\mathbb{R}^4} \text{d} e f^2(q,e) \tilde{G}^{\epsilon}_{R,\delta}(e) =1}} \sum_{\tilde{q},q} \int_{(\mathbb{R}^4)^2} f(\tilde{q},\tilde{e}) \tilde{G}^{\epsilon}_{R,\delta}(\tilde{e}) P_{\tau,M}(\tilde{q} + \tilde{e} - q- e) \tilde{G}^{\epsilon}_{R,\delta}(e) f(q,e) \text{d} e \text{d} \tilde{e}.
\end{equation*}
    Here, $P_{\tau,M}$ is a compactified version of the random walk transition given by
    $$
    P_{\tau,M}(z) = \sqrt{\sum_{l \in \mathbb{Z}^4} P_{\tau}^2(Ml +z)},
    $$
    and $\tilde{G}^{\epsilon}_{R,\delta}$ is a  version of $\tilde{G}_{R,\delta}$ given by
    $$
    \tilde{G}^{\epsilon}_{R,\delta}(e) = \sup_{|d| \le \epsilon} \tilde{G}_{R,\delta}(e+d).
    $$
\end{lem}

In the next section, we will show that $\limsup_{M \to \infty} \limsup_{\epsilon \to 0} \rho_{M,R,\delta,\epsilon} \le \rho$, which will complete the proof of Lemma \ref{lem:uprbnd}.
\begin{proof}

We will have to find an appropriate discretization in order to understand this term carefully. The first step is to write our moment as a norm of a vector in some appropriate vector space and then apply the triangle inequality.
We consider a space of vectors whose entries are indexed by $(l_1,\ldots,l_n) \in (\mathbb{Z}^4)^n$. The norm of such a vector will be given by $\sum_{l_1,\ldots,l_n} (X_{l_1,\ldots,l_n})^2$.  

Now, consider the vector $X^{\rho,e_1,\ldots,e_n}$ whose $l_1,\ldots,l_n$ entry is given by,
\begin{align*}
    &[X^{\rho,e_1,\ldots,e_n}(z_1,\ldots,z_n)]_{l_1,\ldots,l_n} \\
    =& \prod_{i=1}^n \tilde{G}_{R,\delta}(e_{\rho(i)}) P_{\tau}(M l_{\rho(i)} + z_{\rho(i)} +e_{\rho(i)} - M l_{\rho(i-1)} - z_{\rho(i-1)} - e_{\rho(i-1)}).
\end{align*}
Then, we see that we can write, $\mathbb{E}[\mathcal{G}^n]$ as,
\begin{equation*}
    \frac{1}{(n!)^2}\mathbb{E}[(\mathcal{G}_{R,\delta})^n] = \int_{\{(-\frac{M}{2},\frac{M}{2}\}^{4})^{n}} \text{d} z_1\ldots \text{d}z_n  \bigg|\bigg|\frac{1}{n!}\sum_{\rho} \int_{e_1,\ldots,e_n} \text{d}e_1 \ldots \text{d}e_n X^{\rho,e_1,\ldots,e_n}(z_1,\ldots,z_n)\bigg| \bigg|^2.
\end{equation*}
Then, we apply the triangle inequality to state that this is less than,
\begin{equation*}
    \le \int_{\{(-\frac{M}{2},\frac{M}{2}]^4\}^n} \text{d} z_1\ldots \text{d}z_n \left[ \frac{1}{n!} \sum_{\rho} \int_{e_1,\ldots,e_n} ||X^{\rho,e_1,\ldots,e_n}(z_1,\ldots,z_n)|| \right]^2.
\end{equation*}
Recall the definition,
$$
P_{\tau,M}(z) = \sqrt{\sum_{l \in \mathbb{Z}^4} P_{\tau}^2(M l +z)}.
$$
We see that,
\begin{equation*}
    ||X^{\rho,e_1,\ldots,e_n}(z_1,\ldots,z_n)|| = \prod_{i=1}^n \tilde{G}_{R,\delta}(e_{\rho(i)}) P_{\tau,M}(z_{\rho(i)} + e_{\rho(i)} - z_{\rho(i-1)} - e_{\rho(i-1)}).
\end{equation*}
Thus, we see that,
\begin{equation} \label{eq:1stdiscret}
\begin{aligned}
    &\frac{1}{(n!)^2} \mathbb{E}[\mathcal{G}^n]\\& \le \int_{\{(-\frac{M}{2},\frac{M}{2}]^{4}\}^n} \text{d}z_1\ldots \text{d}z_n \left[ \frac{1}{n!} \sum_{\rho} \int_{(\mathbb{R}^4)^n} \text{d}e_1\ldots \text{d}e_n \prod_{i=1}^n \tilde{G}_{R,\delta}(e_{\rho(i)}) P_{\tau,M}(z_{\rho(i)} +e_{\rho(i)} - z_{\rho(i-1)} - e_{\rho(i-1)}) \right]^2.
\end{aligned}
\end{equation}
We still need to discretize the region $(-\frac{M}{2},\frac{M}{2}]^{4}$. 
Fix $\epsilon$ of the form $\frac{M}{2I}$ for some large integer $I$.  
Let $Q_{\epsilon}=(-\epsilon,\epsilon]^4$. Let $P_{\epsilon}$ be a grid of points in $[-\frac{M}{2},\frac{M}{2}]^4$ such that the disjoint union $\cup_{p \in P_{\epsilon}}  p + Q_{\epsilon} =(-\frac{M}{2},\frac{M}{2}]^4 $. 
A quantity that will be useful in trying to understand the discretization would be the following,
\begin{equation*}
    F^{\tilde{G}_{R,\delta}}(z_1,\dots,z_n):= \int_{(\mathbb{R}^4)^n} \text{d}e_1 \ldots \text{d}e_n \prod_{i=1}^n \tilde{G}_{R,\delta}(e_i) P_{\tau,M}(z_i + e_i - z_{i-1} - e_{i-1}).
\end{equation*}

Now, we discuss what happens to the function $F^{\tilde{G}_{R,\delta}}(z_1,\ldots,z_n)$ under a small change to each of its entries $F^{\tilde{G}_{R,\delta}}(z_1 +d_1,\ldots,z_n +d_n)$ where the perturbations $d_i$ are understood to be small, i.e.,d $|d_i| \le \mathfrak{d}$ for some fixed small constant $\mathfrak{d}$. Namely, we see that if we change variable $\hat{e}_i = e_i +d_i$ then an alternative way to write $F^{\tilde{G}_{R,\delta}}(z_1+d_1,\ldots,z_n +d_n)$ would be,
\begin{equation*}
    \int_{(\mathbb{R}^4)^n} \text{d}\hat{e}_1 \ldots \text{d}\hat{e}_n \prod_{i=1}^n\tilde{G}_{R,\delta}(\hat{e}_i - d_i) P_{\tau,M}(z_i + \hat{e}_i - z_{i-1} - \hat{e}_{i-1}).
\end{equation*}
Recall the definition $\tilde{G}^{\mathfrak{d}}_{R,\delta}$ as,
\begin{equation*}
    \tilde{G}^\mathfrak{d}_{R,\delta}(e) = \sup_{|d| \le \mathfrak{d}} \tilde{G}_{R,\delta}(e + d),
\end{equation*}
we thus see that 
$$
F^{\tilde{G}_{R,\delta}}(z_1+d_1,\ldots,z_n+d_n) \le F^{\tilde{G}^{\mathfrak{d}}_{R,\delta}}(z_1,\ldots,z_n),
$$
provided that all $|d_i| \le \mathfrak{d} $. 
Consider the function space $L^2(Q_\epsilon)$ with norm given by,
\begin{equation*}
||f||^2_{L^2(Q_\epsilon)}= \int_{(Q_\epsilon)^n} f^2(z_1,\ldots,z_n) \text{d}z_1\ldots \text{d}z_n.
\end{equation*}
Thus, we can rewrite the right hand side of \eqref{eq:1stdiscret} as,
\begin{equation*}
    \sum_{p_1,\ldots,p_n \in P_{\epsilon}} \left \|\frac{1}{n!}\sum_{\rho}  Y^{\rho}_{p_1,\ldots,p_n}\right\|^2,
\end{equation*}
where $Y^{\rho}_{p_1,\ldots,p_n}$ is the function with values,
$$
Y^{\rho}_{p_1,\ldots,p_n}(z_1,\ldots,z_n) = F^{\tilde{G}_{R,\delta}}(z_{\rho(1)} + p_{\rho(1)},\ldots, z_{\rho(n)} + p_{\rho(n)}).
$$
As before, we apply a slightly different triangle inequality to deduce that 
\begin{equation*}
\begin{aligned}
     & \frac{1}{(n!)^2}\mathbb{E}[(\mathcal{G}_{R,\delta})^n ]  \le \sum_{p_1,\ldots,p_n \in P_{\epsilon}} \left[ \frac{1}{n!} \sum_{\rho} ||Y^{\rho}_{p_1,\ldots,p_n}|| \right]^2\\
     &= (\epsilon^{4n})\sum_{p_1.\ldots,p_n \in P_{\epsilon}} \left[ \frac{1}{n!} \sum_{\rho} \left(\frac{1}{(\epsilon)^{4n}}\int_{[-\epsilon/2,\epsilon/2]^4} \text{d} d_1 \ldots \text{d} d_n  F^{\tilde{G}_{R,\delta}}(p_{\rho(1)} + d_1,\ldots,p_{\rho(n)}+d_n)^2 \right)^{1/2}\right]^2\\
     &\le \epsilon^{4n} \sum_{p_1,\ldots,p_n \in P_{\epsilon}} \left[ \frac{1}{n!} \sum_{\rho} F^{\tilde{G}^{\epsilon}_{R,\delta}}(p_{\rho(1)},\ldots, p_{\rho(n)}) \right]^2\\
     &= \epsilon^{4n} \sum_{p_1,\ldots,p_n \in P_{\epsilon}}\left[\frac{1}{n!} \sum_{\rho} \int_{(\mathbb{R}^4)^n}  \text{d}e_1 \ldots \text{d}e_n  \prod_{i=1}^n \tilde{G}^{\epsilon}_{R,\delta}(e_{\rho(i)}) P_{\tau,M}(p_{\rho(i)} + e_{\rho(i)} - p_{\rho(i-1)} - e_{\rho(i-1)}) \right]^2.
\end{aligned}
\end{equation*}
Let us consider the term inside the brackets. Consider the point measure $\mu$ given by,
\begin{equation*}
    \mu_{p} =\frac{1}{n} \sum_{i=1}^n \delta_{p_i},
\end{equation*}
thus, we have a point measure supported at each point $p_i$. 
Related to the measure $\mu$, we can also define the following function on the points $q$ of $P_{\epsilon}$: 
$$
\phi_{\mu}(p) = \sqrt{\mu(p)}.
$$
This function is normalized so that,
\begin{equation*}
    \sum_{p} (\phi_{\mu}(p))^2 = 1.
\end{equation*}
We thus have that,
\begin{equation*}
\begin{aligned}
    &\frac{1}{n!}\sum_{\rho} \int_{(\mathbb{R}^4)^n} \text{d}e_1\ldots \text{d}e_n \prod_{i=1}^n \tilde{G}^{\epsilon}_{R,\delta}(e_{\rho(i)}) P_{\tau,M}(p_{\rho(i)} + e_{\rho(i)} -p_{\rho(i-1)} -e_{\rho(i-1)})\\
    &= \frac{1}{n!}\sum_{\rho}\sum_{q_1,\ldots,q_n} \mathbbm{1}(p_{\rho(i)} = q_i, \forall i) \int_{(\mathbb{R}^4)^n}\text{d}e_1\ldots \text{d}e_n \prod_{i=1}^n 
    \tilde{G}^{\epsilon}_{R,\delta}(e_i) P_{\tau,M}(q_i + e_i - q_{i-1} - e_{i-1})\\
    &= \frac{1}{n!} \sum_{q_1,\ldots,q_n} \mathbbm{1}(\mu_p = \mu_q) \prod_{r \in P_{\epsilon}} (n\mu_p(r))! \int_{(\mathbb{R}^4)^n} \text{d}e_1\ldots \text{d}e_n
 \prod_{i=1}^n  \tilde{G}^{\epsilon}_{R,\delta}(e_i) P_{\tau,M}(q_i + e_i -q_{i-1} - e_{i-1})\\
 &= \frac{1}{n!} \sum_{q_1,\ldots,q_n} \mathbbm{1}(\mu_p = \mu_q) \prod_{r \in P_{\epsilon}} \frac{(n\mu_p(r))!}{(\phi_{\mu_p}(r))^{n \phi_{\mu_p}(r)}} \\
 & \times\int_{(\mathbb{R}^4)^n} \text{d}e_1\ldots \text{d}e_n \sqrt{\phi_{\mu_p(q_n)}}
 \prod_{i=2}^n \sqrt{\phi_{\mu_p}(q_{i})} \tilde{G}^{\epsilon}_{R,\delta}(e_i) P_{\tau,M}(q_i + e_i -q_{i-1} - e_{i-1}) \sqrt{\phi_{\mu_p}(q_{i-1})} \\
 & \hspace{5 cm}\times \tilde{G}^{\epsilon}_{R,\delta}(e_1) \sqrt{\phi_{\mu_p}(q_1+e_1)} P_{\tau,M}(q_1 + e_1)
  \end{aligned}
\end{equation*}
and it is bounded by
 \begin{equation}\label{eq:upperbndexp}
\begin{aligned}
 &  [\max_z P_{\tau,M}(z)] \frac{1}{n!} \prod_{r \in P_{\epsilon}} \frac{(n\mu_p(r))!}{(\phi_{\mu_p}(r))^{n \mu_p(r)}}\\
 & \times \sum_{q_1,\ldots,q_n}\int_{(\mathbb{R}^4)^n}  \text{d}e_1\ldots \text{d}e_n \sqrt{\phi_{\mu}(q_n)}\prod_{i=2}^n \sqrt{\phi_{\mu_p}(q_i)} \tilde{G}^{\epsilon}_{R,\delta}(e_i) P_{\tau,M}(q_i + e_i - q_{i-1} - e_{i-1}) \sqrt{\phi_{\mu_p}(q_{i-1})}\\
 &\hspace{5 cm} \times  \tilde{G}^{\epsilon}_{R,\delta}(e_1)  \sqrt{\phi_{\mu_p}(q_1+e_1)}.
 \end{aligned}
\end{equation}
We can, again, represent the last line as an operator computation. 
Consider the following space of functions,
\begin{equation*}
\begin{aligned}
    & L^2_{G,R,\delta,\epsilon}:= \left\{f:  \epsilon^4 \sum_{q} \int_{\mathbb{R}^4} f^2(q,e) \tilde{G}^{\epsilon}_{R,\delta}(e) \text{d}e =1\right\},\\&
    \langle f_1, f_2 \rangle= \epsilon^4 \sum_{q} \int_{\mathbb{R}^4} f_1(q,e)f_2(q,e) \tilde{G}^{\epsilon}_{R,\delta}(e) \text{d} e. 
\end{aligned}
\end{equation*}
On this space, we consider the following operator,
\begin{equation*}
    T_{k,R,\delta,\epsilon} (f)(\tilde{q},\tilde{e})=\sqrt{k}(\tilde{q}) \sum_{q}\int_{\mathbb{R}^4} \text{d}e P_{\tau,M}(\tilde{q}+\tilde{e}-q-e) \tilde{G}^{\epsilon}_{R,\delta}(e)\sqrt{k}(q) .
\end{equation*}
This is a symmetric operator on our space $L^{2}_{G,R,\delta,\epsilon}$. 
We can rewrite the last line of \eqref{eq:upperbndexp} as,
\begin{equation*}
    [\max_{z} P_{\tau,M}] \frac{1}{n!} \prod_{r\in P_{\epsilon}} \frac{(n \mu_p(r))!}{(\phi_{\mu_p}(r))^{n \mu_p(r)}} \left[ \int_{\mathbb{R}^4} \text{d}e \tilde{G}^{\epsilon}_{R,\delta}(e) \right]\bigg\langle \sqrt{\frac{\phi_{\mu_p}}{\int_{\mathbb{R}^4} \text{d} e \tilde{G}^{\epsilon}_{R,\delta}(e)}}, T_{\phi_{\mu_p},R,\epsilon}^{n-1} \sqrt{\frac{\phi_{\mu_p}}{\int_{\mathbb{R}^4} \text{d} e \tilde{G}^{\epsilon}_{R,\delta}(e)}} \bigg \rangle.
\end{equation*}
We needed to introduce the normalization factor $\int_{\mathbb{R}^4} \text{d}e \tilde{G}^{\epsilon}_{R,\delta}(e)$ so that the inner product of the function $\sqrt{\frac{\phi_{\mu_p}}{\int_{\mathbb{R}^4} \text{d} e \tilde{G}^{\epsilon}_{R,\delta}(e)}}$ with itself has norm less than $1$. Observe that,
\begin{equation*}
    \sum_{q} \int_{\mathbb{R}^4}\left[\sqrt{\frac{\phi_{\mu_p}(q)}{\int_{\mathbb{R}^4} \text{d} e \tilde{G}^{\epsilon}_{R,\delta}(e)}}\right]^2 \tilde{G}^{\epsilon}_{R,\delta}(e)  \text{d} e 
    = \sum_{q} \phi_{\mu_p}(q) = \sum_{q}\sqrt{\mu_p(q)} \le \sum_{q} \mu_p(q) = 1.
\end{equation*}
As restriction of the domain to $R$ is needed in order to ensure that  $\int_{\mathbb{R}^4} \tilde{G}^{\epsilon}|_R(e) \text{d}e$ is finite. 
The inner product can be bounded as,
\begin{equation*}
\begin{aligned}
 & \left\langle \sqrt{\frac{\phi_{\mu_p}}{\int_{\mathbb{R}^4} \text{d} e \tilde{G}^{\epsilon}_{R,\delta}(e)}}, T_{\phi_{\mu_p},R,\epsilon}^{n-1} \sqrt{\frac{\phi_{\mu_p}}{\int_{\mathbb{R}^4} \text{d} e \tilde{G}^{\epsilon}_{R,\delta}(e)}}  \right\rangle\\
 &\le \left[\max_{\substack{ \sum_q k^2(q) = 1\\
 \sum_q \int_{\mathbb{R}^4} \text{d} e f^2(q,e) \tilde{G}^{\epsilon}_{R,\delta}(e) =1}} \sum_{\tilde{q},q}
 \int_{(\mathbb{R}^4)^2}  f(\tilde{q},\tilde{e}) \tilde{G}^{\epsilon}_{R,\delta}(\tilde{e}) P_{\tau,M}(\tilde{q} + \tilde{e} - q- e) \tilde{G}^{\epsilon}_{R,\delta}(e) f(q,e) \text{d} e \text{d} \tilde{e} \right]^{n-1}.
\end{aligned}
\end{equation*}
We denote the quantity in brackets above by $\rho_{M,R,\delta,\epsilon}$.

Returning to bounding $\mathbb{E}[(\mathcal{G}_{R,\delta})^{n}]$, we see that this is bounded by,
\begin{align*}
    &\frac{1}{(n!)^2}\mathbb{E}[(\mathcal{G}_{B,R})^n] \\
    \le&|\max_{z} P_{\tau,M}(z)|^2 \left[ \int_{\mathbb{R}^4} \tilde{G}^{\epsilon}_{R,\delta}(e)  \text{d} e  \right]^2 (\rho_{M,R,\delta,\epsilon})^{2n-2} \sum_{p_1,\ldots,p_n} \left(\frac{1}{n!} \prod_{r \in P_{\epsilon}} \frac{(n \mu_p(r))!}{(\phi_{\mu_p}(r))^{n \mu_p(r)}} \right)^2.
\end{align*}
Then, we see that,
\begin{equation*}
\limsup_{n \to \infty}\frac{1}{n} \log \frac{1}{(n!)^2} \mathbb{E}[(\mathcal{G}_{R,\delta})^n] \le 2 \log \rho_{M,R,\delta,\epsilon} + \frac{1}{n} \log \sum_{p_1,\ldots,p_n} \left(\frac{1}{n!} \prod_{r \in P_{\epsilon}} \frac{(n \mu_p(r))!}{(\phi_{\mu_p}(r))^{n \mu_p(r)}} \right)^2.
\end{equation*}
The latter term above can be shown to go to 0. If we note equation \cite[(3.1.11)]{Chenbook}, 
we see that an upper bound on the logarithmically scaled moments of $\mathcal{G}_{R,\delta}$ is bounded by $2 \rho_{M,R,\delta,\epsilon}$.
\end{proof}

\subsection{Analyzing $\rho_{M,R,\delta,\epsilon}$}




The goal of this section is to remove the dependence of $\epsilon$ and $M$ in the definition of the optimization $\rho_{M,R,\delta,\epsilon}$.  We will prove the following two Lemmas. The first will remove the dependence on $\epsilon$. The second will remove the dependence on $M$.

\begin{lem} \label{lem:removeepsilon}
Recall $\rho_{M,R,\delta,\epsilon}$ from Lemma \ref{lem:uprBRdeltaepsi}. 
As we remove the $\epsilon$ regularization, we argue that
\begin{equation*}
\limsup_{\epsilon \to 0} \rho_{M,R,\delta,\epsilon} \le \rho_{M,R,\delta}.
\end{equation*}
Here,
\begin{equation*}
\begin{aligned}
    \rho_{M,R,\delta}:= & \sup_{\substack{\sum_q k^2(q) =1\\
    \int_{(-\frac{M}{2},\frac{M}{2}]^4} \text{d}q\int_{\mathbb{R}^4} \text{d}e f^2(q,e) \tilde{G}_{R,\delta}(e) =1}} 
    \int_{((-\frac{M}{2},\frac{M}{2}]^4)^2}  \text{d}\tilde{q} \text{d}q \int_{(\mathbb{R}^4)^2} \text{d} \tilde{e} \text{d} e f(\tilde{q},\tilde{e}) \sqrt{k}(\tilde{q}) \tilde{G}_{R,\delta}(\tilde{e}) \\
    & \hspace{5cm} \times P_{\tau,M}(\tilde{q}+ \tilde{e} -q -e) \tilde{G}_{R,\delta}(e) \sqrt{k}(q) f(q,e).
\end{aligned}
\end{equation*}
\end{lem}

\begin{lem} \label{lem:removM}
    Recall $\rho_{M,R,\delta}$ from Lemma \ref{lem:removeepsilon}.
    As we remove the $M$ compactification, we have,
    \begin{equation*}
        \limsup_{M \to \infty} \rho_{M,R,\delta} \le \rho.
    \end{equation*}
\end{lem}

These two lemmas are now enough to prove Lemma \ref{lem:uprbnd}.
\begin{proof} [Proof of Lemma \ref{lem:uprbnd}]
We have that from Lemma \ref{lem:uprBRdeltaepsi} that the  asymptotic moments of $\mathcal{G}_{B,R}$ are bounded by $\rho_{M,R,\delta,\epsilon}$ for any arbitrary choice of $M$ and $\epsilon$. By using Lemmas \ref{lem:removeepsilon} and \ref{lem:removM} we can take the limit as $\epsilon \to 0$ and $\delta \to 0$ in order to deduce that the asymptotic moments of $\rho_{M,R,\delta,\epsilon}$ can be bounded by $\rho$, as desired.
\end{proof}

Now we can turn to the proofs of Lemmas \ref{lem:removeepsilon} and \ref{lem:removM}.

\begin{proof}[Proof of  Lemma \ref{lem:removeepsilon}]
Note that
$\rho_{M,R,\delta,\epsilon}$ corresponds to the maximization problem,
\begin{equation*}
\epsilon^{2d} \sum_{z_1,z_2 \in P_{\epsilon}} \int_{(\mathbb{R}^4)^2 } \text{d}e_1 \text{d}e_2 f(z_1,e_1) \sqrt{k(z_1)} G_{R,\delta}^{\epsilon}(e_1) P_{\tau,M}(z_1 + e_1 -z_2 -e_2)  G_{R,\delta}^{\epsilon}(e_2) \sqrt{k(z_2)} f(z_2,e_2). 
\end{equation*}
Fix some $\epsilon_0$, for $\epsilon \le \epsilon$, we can find some function $f(\epsilon_0)$ such that $f(\epsilon_0) \to 1$ as $\epsilon_0 \to 0$. Furthermore,
$G_{R,\delta}^{\epsilon}(z) \le f(\epsilon_0) G_{R,\delta}(z) + \mathbbm{1}[R \le |z| \le R + \epsilon_0]$.
Notice that once we fix $\epsilon_0$, we can apply Theorem \ref{thm:oanaleps} to the function $f(\epsilon_0)G_{R,\delta}(z) +\mathbbm{1}[R \le |z| \le R + \epsilon_0] $ and show that, 
\begin{equation} \label{eq:someupper}
\begin{aligned}
& \lim_{\epsilon \to 0}\sup_{f,k} \epsilon^{2d} \sum_{z_1,z_2 \in P_{\epsilon}} \int_{(\mathbb{R}^4)^2 } \text{d}e_1 \text{d}e_2 f(z_1,e_1) \sqrt{k(z_1)} G_{R,\delta}^{\epsilon}(e_1) P_{\tau,M}(z_1 + e_1 -z_2 -e_2)  G_{R,\delta}^{\epsilon} \sqrt{k(z_2)} f(z_2,e_2)\\
&\le \sup_{f,k} \int_{([-M,M]^4)^2} \text{d}z_1 \text{d}z_2 \int_{(\mathbb{R}^4)^2} \text{d}e_1 \text{d}e_2 f(z_1,e_1) \sqrt{k(z_1)} [f(\epsilon_0) G_{R,\delta}(e_1) + \mathbbm{1}[R \le |e_1| \le R + \epsilon_0]]\\
& \times P_{\tau,M}(z_1 + e_1 -z_2 -e_2) [f(\epsilon_0) G_{R,\delta}(e_2) + \mathbbm{1}[R \le |e_2| \le R + \epsilon_0]] \sqrt{k(z_2)} f(z_2,e_2).
\end{aligned}
\end{equation}

Now, we assert that in general, we have for any $L>0$ and positive functions, $f,M_1$ and $M_2$ that
\begin{equation} \label{eq:splitcross}
\begin{aligned}
&\int_{([-M,M]^4)^2} \text{d}z_1 \text{d}z_2 \int_{(\mathbb{R}^4)^2} \text{d}e_1 \text{d}e_2 k(z_1,e_1) M_1(e_1) P_{\tau,M}(z_1 +e_1 - z_2 -e_2) k(z_2,e_2) M_2(e_2) \\
&\le
L \int_{([-M,M]^4)^2} \text{d}z_1 \text{d}z_2 \int_{(\mathbb{R}^4)^2} \text{d}e_1 \text{d}e_2 k(z_1,e_1) M_1(e_1) P_{\tau,M}(z_1 +e_1 - z_2 -e_2) k(z_2,e_2) M_1(e_2) \\
& + \frac{1}{L} \int_{([-M,M]^4)^2} \text{d}z_1 \text{d}z_2 \int_{(\mathbb{R}^4)^2} \text{d}e_1 \text{d}e_2 k(z_1,e_1) M_2(e_1) P_{\tau,M}(z_1 +e_1 - z_2 -e_2) k(z_2,e_2) M_2(e_2). 
\end{aligned}
\end{equation}
To see this, we introduce the convolutional square root of $P_{\tau,M}(z_1 +e_1 - z_2 - e_2) = \int_{[-M,M]^4} \text{d}k \tilde{P}_{\tau,M}(z_1 + e_1 - k) \tilde{P}_{\tau,M}(k- z_2 -e_2) $. Observe that $\tilde{P}_{\tau,M}(y) = \tilde{P}_{\tau,M}(-y)$ by symmetry. 
Thus, we have that
\begin{equation*}
\begin{aligned}
&\int_{([-M,M]^4)^2} \text{d}z_1 \text{d}z_2 \int_{(\mathbb{R}^4)^2} \text{d}e_1 \text{d}e_2 k(z_1,e_1) M_1(e_1) P_{\tau,M}(z_1 +e_1 - z_2 -e_2) k(z_2,e_2) M_2(e_2)\\
&= \int_{[-M,M]^4} \text{d} k \left[\int_{[-M,M]^4} \text{d}z_1  \int_{\mathbb{R}^4} \text{d}e_1  k(z_1,e_1) M_1(e_1) \tilde{P}_{\tau,M}(z_1 +e_1-k) \right]\\
& \times \left[\int_{[-M,M]^4} \text{d}z_2  \int_{\mathbb{R}^4} \text{d}e_2  k(z_2,e_2) M_2(e_2) \tilde{P}_{\tau,M}(z_2 +e_2-k) \right]
\end{aligned}
\end{equation*}
and it is bounded by
\begin{equation*}
\begin{aligned}
&  L \int_{[-M,M]^4} \text{d} k \left[\int_{[-M,M]^4} \text{d}z_1  \int_{\mathbb{R}^4} \text{d}e_1  k(z_1,e_1) M_1(e_1) \tilde{P}_{\tau,M}(z_1 +e_1-k) \right]^2\\
&+ \frac{1}{L}\int_{[-M,M]^4} \text{d} k \left[\int_{[-M,M]^4} \text{d}z_2  \int_{\mathbb{R}^4} \text{d}e_2  k(z_2,e_2) M_2(e_2) \tilde{P}_{\tau,M}(z_2 +e_2-k) \right]^2\\
& =  L \int_{([-M,M]^4)^2} \text{d}z_1 \text{d}z_2 \int_{(\mathbb{R}^4)^2} \text{d}e_1 \text{d}e_2 k(z_1,e_1) M_1(e_1) P_{\tau,M}(z_1 +e_1 - z_2 -e_2) k(z_2,e_2) M_1(e_2) \\
& + \frac{1}{L} \int_{([-M,M]^4)^2} \text{d}z_1 \text{d}z_2 \int_{(\mathbb{R}^4)^2} \text{d}e_1 \text{d}e_2 k(z_1,e_1) M_2(e_1) P_{\tau,M}(z_1 +e_1 - z_2 -e_2) k(z_2,e_2) M_2(e_2). 
\end{aligned}
\end{equation*}
Applying equation \eqref{eq:splitcross} to the last line of \eqref{eq:someupper}, we can bound the last line by,
$$
\begin{aligned}
&[f(\epsilon_0)^2 + L] \int_{(\mathbb{R}^4)^2} \text{d}z_1 \text{d}z_2 \int_{(\mathbb{R}^4)^2 } \text{d}e_1 \text{d}e_2 f(z_1,e_1) \sqrt{k(z_1)} G_{R,\delta}(e_1) P_{\tau,M}(z_1 + e_1 -z_2 -e_2)  \\
& \times G_{R,\delta}(e_2) \sqrt{k(z_2)} f(z_2,e_2) 
\\&+[L^{-1} + 1]\int_{(\mathbb{R}^4)^2} \text{d}z_1 \text{d}z_2 \int_{(\mathbb{R}^4)^2 } \text{d}e_1 \text{d}e_2 f(z_1,e_1) \sqrt{k(z_1)} \mathbbm{1}[R \le |e_1| \le R + \epsilon_0] \\
& \times P_{\tau,M}(z_1 + e_1 -z_2 -e_2)  \mathbbm{1}[R \le |e_1| \le R + \epsilon_0] \sqrt{k(z_2)} f(z_2,e_2). 
\end{aligned}
$$
The final term on the last line below can be bounded from above by  
$\sup_z  \int_{\mathbb{R}^4} \mathbbm{1}[R\le |z-y| \le R+ \epsilon_0] \mathbbm{1}[R \le |y| \le R+\epsilon_0] \text{d}y \le [R+\epsilon_0]^4 - R^4$. This is a consequence of the lower bound from Section \ref{sec:lwrbndbrown}. 
If we now first take $\epsilon_0 \to 0$ and then finally $L \to 0$, this gives us our desired conclusion from Lemma \ref{lem:removeepsilon}.
\end{proof}

Now, we turn a sketch of the proof of Lemma \ref{lem:removM}.
\begin{proof}[Proof of  Lemma \ref{lem:removM}]
We omit the proof since the proof is very similar to that to  \cite[Lemma 3.2.4]{Chenbook}.   
\end{proof}

\section{The relationship between $\rho$ and the modified Gagliardo-Nirenberg constant}\label{sec:rhoGagnir}

The goal of this section is to show that constant $\rho$ which shown determines the large deviation behavior of $\mathcal{G}$ can be more simply represented as a constant that occurs more naturally in analysis. Namely, the modified Gagliardo-Nirenberg constant as in \cite[Equation (6)]{FY15}. 

Before we present our main theorem, we discuss some notation. 
Recall that we let $p_t(x)$ be the transition density for a Brownian motion to reach point $x$ at time $t$ and $G(x)=\int p_t(x) dt$, $\tilde{G}$ be the convolutional square root of $G$, so that $\tilde{G}*\tilde{G}=G$ and $P_{\tau}(x)=\int e^{-t}p_t(x) dt$.  
\begin{prop} \label{prop:GagNir}
Recall the optimization problem:
\begin{equation*}
\rho:= \sup_{\substack{f \in L^2_G \\ k: \int_{\mathbb{R}^4} k^2(z) \text{d}z =1}} 
\int_{(\mathbb{R}^4)^4} f(\tilde{z},\tilde{e}) \sqrt{k}(\tilde{z}) \tilde{G}(\tilde{e})P_{\tau} (\tilde{z}+\tilde{e}-z-e) \tilde{G}(e) \sqrt{k}(z) f(z,e) \text{d}\tilde{z} \text{d} \tilde{e} \text{d}z \text{d}e.
\end{equation*}
Let $\tilde{\kappa}(4,2)$ be the optimal constant in the modified Gagliardo-Nirenberg inequality. 
Namely, the best constant such that,
\begin{equation*}
\left(\int_{(\mathbb{R}^4)^2} g^2(x) G(x-y) g^2(y) \text{d}x \text{d}y \right)^{1/4} \le \tilde{\kappa}(4,2) ||g||^{1/2}_{L^2} ||\nabla g||^{1/2}_{L^2}.
\end{equation*}
Then, 
\begin{equation*}
\rho=\frac{\tilde{\kappa}^2(4,2)}{\sqrt{2}}. 
\end{equation*}
\end{prop}

\begin{proof}
\textit{Part 1: Showing $\rho \ge \frac{\tilde{\kappa}^2(4,2)}{\sqrt{2}}$}



First, we show that $\rho$ is greater than the value of a certain optimization problem, which can more readily be shown to related to $\tilde{\kappa}(4,2)$. 
Let $h$ be a function such that,
\begin{equation*}
\int_{(\mathbb{R}^4)^2}h^2(x) G(x-y) h^2(y) \text{d}x \text{d}y =1. 
\end{equation*}
Substitute $k(x) = \int_{\mathbb{R}^4} h^2(y) \tilde{G}(x-y) \text{d}y$ and we find $f(x,e) = h(x+e) \sqrt{\int_{\mathbb{R}^4} h^2(y) \tilde{G}(x-y) \text{d}y } $. 
Indeed, 
$$
\begin{aligned}
\int_{\mathbb{R}^4} k^2(x) \text{d}x &= \int_{(\mathbb{R}^4)^3}g^2(z_1) \tilde{G}(x-z_1)  \tilde{G}(z_2 -x) g^2(z_2) \text{d}x \text{d}z_1 \text{d}z_2\\
&= \int_{(\mathbb{R}^4)^2} g^2(z_1) G(z_2 -z_1) g^2(z_2) \text{d}z_1 \text{d}z_2 =1. 
\end{aligned}
$$
In addition,
$$
\begin{aligned}
 \int_{(\mathbb{R}^4)^2} f^2(x,e) \tilde{G}(e) \text{d}x \text{d}e 
&= \int_{(\mathbb{R}^4)^3} h^2(x+e) \tilde{G}(e) \tilde{G}(x-y) h^2(y) \text{d}x \text{d}y \text{d}e \\ 
&
=\int_{(\mathbb{R}^4)^2} \text{d}x \text{d}y h^2(x) h^2(y) G(x-y) =1. 
\end{aligned}
$$
If we let $J(x)=h(x) \int_{\mathbb{R}^4} h^2(x+\psi)G (\psi) \text{d}\psi$, then 
\begin{equation}\label{eq:optimizationprobleminequ}
\rho \ge \sup
\int_{(\mathbb{R}^4)^2} J(x) P_\tau (x-y) J(y) \text{d}x \text{d}y, 
\end{equation}
where the supremum is taken over all functions satisfying 
\begin{equation*}
\int_{(\mathbb{R}^4)^2} h(x)^2 G (x-y) h(y)^2 \text{d}x \text{d}y =1. 
\end{equation*}
Let 
\begin{equation*}
M(\theta)
=\sup_{\substack{g: \int_{\mathbb{R}^4} g^2 \text{d}x =1
\\ \int_{\mathbb{R}^4} |\nabla g|^2  \text{d}x< \infty}} \theta \left(  \int_{(\mathbb{R}^4)^2} g^2(x) G (x-y) g^2(y)  \text{d}x \text{d}y \right)^{1/2} 
- \frac{1}{2}\int_{\mathbb{R}^4}  |\nabla g|^2 \text{d}z. 
\end{equation*}
Then, 
\begin{equation*}
M\left(\frac{1}{\rho}\right)
=\sup_{\substack{ g: \int_{\mathbb{R}^4} g^2 \text{d}x =1
\\ \int_{\mathbb{R}^4} |\nabla g|^2  \text{d}x< \infty}} \frac{1}{\rho} \left(   \int_{(\mathbb{R}^4)^2} g^2(x) G (x-y) g^2(y) \text{d}x \text{d}y  \right)^{1/2} 
- \frac{1}{2}\int_{\mathbb{R}^4}  |\nabla g|^2 \text{d}z. 
\end{equation*}
Our two intermediate goals are to first show that $M(\rho^{-1}) = \frac{\tilde{\kappa}^4(4,2)}{2\rho^2}$ and secondly to show that $M(\rho^{-1}) \le 1.$ Together, these imply $\rho \ge \frac{\tilde{\kappa}^2(4,2)}{\sqrt{2}}$. 
We can first check that, by the modified Gagliardo-Nirenberg inequality, that, for any function $g$ with $\int_{\mathbb{R}^4} g^2 \text{d}x =1$, we have 
\begin{equation} \label{eq:maxirho}
\begin{aligned}
& \frac{1}{\rho} \left(  \int_{(\mathbb{R}^4)^2} g^2(x) G (x-y) g^2(y) \text{d}x \text{d}y \right)^{1/2} -\frac{1}{2} \int_{\mathbb{R}^4}  |\nabla g|^2 \text{d}z\\
& \le \frac{\tilde{\kappa}^2(4,2)}{\rho}  \left[\int_{\mathbb{R}^4}  |\nabla g|^2\text{d} z \right]^{1/2} -\frac{1}{2} \int_{\mathbb{R}^4}  |\nabla g|^2 \text{d} z \\
&\le \frac{\tilde{\kappa}^4(4,2)}{2\rho^2} + \frac{1}{2}\int_{\mathbb{R}^4}  |\nabla g|^2\text{d} z -\frac{1}{2} \int_{\mathbb{R}^4}  |\nabla g|^2 \text{d} z  = \frac{\tilde{\kappa}^4(4,2)}{2\rho^2}.
\end{aligned}
\end{equation}

Now we show that there is a function $f$  such that the supremum $\frac{\tilde{\kappa}^4(4,2)}{2\rho^2}$ is actually attained. 
From \cite[Theorem 2.2]{FY15}, we know that there is a function $\tilde{f}$ that satisfies the equality conditions in the modified Gagliardo-Nirenberg inequality such that its $L^2$ norm is $1$. 
Consider the rescaled version $f^{\lambda} = \lambda^{2} \tilde{f}(\lambda x)$. 
This transformation preserves the $L^2$ norm while $||\nabla f^{\lambda}||_{L^2} = \lambda ||\nabla \tilde{f}||_{L^2}$. 
One can also check that,
$$
\left[\int_{(\mathbb{R}^4)^2} (f^{\lambda}(x))^2 G(x-y) (f^{\lambda}(y))^2 \text{d}x \text{d}y \right]^{1/4} 
= \lambda^{1/2}\left[\int_{(\mathbb{R}^4)^2} (\tilde{f})^2 G(x-y) (\tilde{f})^2 \text{d}x \text{d}y \right]^{1/4}.
$$
By appropriately tuning $\lambda$, one can check that all inequalities in \eqref{eq:maxirho} become equalities and the maximum is attained. Let $f$ be denote the function at which this supremum is attained.  This proves the equality $M(\rho^{-1})$ with $\frac{\tilde{\kappa}^4(4,2)}{2\rho}$.

Now, we turn to showing that $M(\rho^{-1}) \le 1$. 
This involves manipulating the function $f$ at which the maximum is attained carefully though the use of Lagrange multipliers. 
By the Lagrange multiplier condition for the supremum of $M(\rho^{-1})$, 
\begin{equation*}
\frac{1}{\rho} \frac{f(x) \int_{\mathbb{R}^4} G (x-y) f(y)^2 \text{d}y}
{[\int_{(\mathbb{R}^4)^2} f^2(x) G (x-y) f^2(y) \text{d}x \text{d}y]^{1/2}}
+\frac{1}{2} \Delta f(x)
=M(\rho^{-1}) f(x). 
\end{equation*}
Let 
\begin{equation*}
\overline{f}
= \frac{f}{[\int_{(\mathbb{R}^4)^2} f^2(x) G (x-y) f^2(y) \text{d}x \text{d}y]^{1/4}}. 
\end{equation*}
Then, we obtain
\begin{equation*}
\frac{1}{\rho} \overline{f}(x) \int_{\mathbb{R}^4} G (x-y) \overline{f}^2(y) \text{d}y
+\frac{1}{2} \Delta \overline{f}(x)
=M(\rho^{-1}) \overline{f}(x), 
\end{equation*}
where the normalization is set by
\begin{equation*}
\int_{(\mathbb{R}^4)^2} \overline{f}^2(x)  G (x-y) \overline{f}^2(y) \text{d}x\text{d}y 
=1.
\end{equation*}
Let $W(x)=\overline{f}(x) \int_{\mathbb{R}^4} \overline{f}^2(y)G(x-y) \text{d}y$. 
Then, 
\begin{equation*}
\int_{\mathbb{R}^4} \frac{1}{\rho} W(x) P_{\tau}W(x) \text{d}x
+\int_{\mathbb{R}^4} \frac{1}{2} \Delta \overline{f}(x)P_{\tau}W(x) \text{d}x
= \int_{\mathbb{R}^4} M(\rho^{-1}) \overline{f}(x)P_{\tau}W(x) \text{d}x, 
\end{equation*}
where $P_{\tau}W(x)=\int_{\mathbb{R}^4} P_{\tau}(y)W(x-y)\text{d}y$. 
Since $\int_{\mathbb{R}^4} W(x) P_{\tau}W(x) \text{d}x \le \rho$  by the optimization problem inequality \eqref{eq:optimizationprobleminequ}, 
\begin{equation*}
1 +\int_{\mathbb{R}^4} \frac{1}{2} \Delta \overline{f}(x)P_{\tau}W(x) \text{d}x
\ge M(\rho^{-1})\int_{\mathbb{R}^4}  \overline{f}(x)P_{\tau}W(x) \text{d}x. 
\end{equation*}
Then, since $P_{\tau}=I+2^{-1}\Delta \circ P_{\tau}$ and $\int_{\mathbb{R}^4} \overline{f}(x)W(x)\text{d}x=1$ by the normalization condition on $\overline{f}$,
\begin{align*}
\frac{1}{2}\int_{\mathbb{R}^4} \Delta \overline{f}(x)P_{\tau}W(x) \text{d}x
=& \frac{1}{2}\int_{\mathbb{R}^4} \overline{f} (x) \Delta P_{\tau}W(x)\text{d}x \\
=& -\int_{\mathbb{R}^4} \overline{f}(x)W(x)dx + \int_{\mathbb{R}^4} \overline{f}(x) P_{\tau}W(x) \text{d}x\\
=&-1+\int_{\mathbb{R}^4} \overline{f}(x) P_{\tau}W(x) \text{d}x. 
\end{align*}
Therefore, $1\ge M(\rho^{-1})$ 
and hence $\rho \ge \frac{\tilde{\kappa}(4,2)^4}{\sqrt{2}}$. 

\textit{Part 2: Showing $\rho \le \frac{\tilde{\kappa}^2(4,2)}{\sqrt{2}}$}

If we recall the problem $M(\rho^{-1})$, showing that $\rho \le \frac{\tilde{\kappa}^2(4,2)}{\sqrt{2}}$ is ultimately equivalent to showing $M(\rho^{-1}) \ge 1$. To do this, it suffices to find a good candidate function for the optimization problem defining $M(\rho^{-1})$.  We find the proposed candidate function by considering the maximizer of the following auxiliary function. 
Consider the following problem: 
\begin{align*}
c_0
:= \inf 
\{ \int_{\mathbb{R}^4} f^2(x) dx +\frac{1}{2}\int_{\mathbb{R}^4} |\nabla f(x)|^2 \text{d}x 
\text{ s.t. }
\int_{(\mathbb{R}^4)^2} f^2(x) G (x-y) f^2(y) \text{d}x \text{d}y =1
\}.
\end{align*}
We  will first argue that  $\rho \le c_0^{-1}$.  As we will discuss in more detail in equation \eqref{eq:auxtomainopt}, the intuition regarding the main relationship between $c_0$ and $M(\theta)$ is that $M(\theta)$ will exactly be $1$ when $\theta = c_0$ (or more exactly that $M(\theta)>1$ if $\theta >c_0$). 
Thus, we will be done if we show  $\rho \le c_0^{-1}$. 
Given $f$ and $k$ such that,
$$
\int_{(\mathbb{R}^4)^2} f^2(z,e) \tilde
{G}(e) \text{d}z \text{d}y  =1, \quad \int_{\mathbb{R}^4} k^2(z) \text{d} z = 1,
$$
and consider 
$$
F(\lambda):= \int_{\mathbb{R}^4} f(\lambda-e,e) \sqrt{k}(\lambda-e) \tilde{G}(e) \text{d} e.
$$
It suffices to show 
\begin{align*}
A := \int_{(\mathbb{R}^4)^2}  P_{\tau}(x-y) F(x) F(y) \text{d}x \text{d}y \le c_0^{-1}.
\end{align*} 
We remark here that the quantity on the right hand side of the definition symbol is exactly the term in the optimization problem defining $\rho$ when we use test functions $f$ and $k$ from earlier. 

We first claim that for any $h$, we have that,
$$
\int_{\mathbb{R}^4} F(x) h(x) \text{d} x \le \left[\int_{(\mathbb{R}^4)^2} h^2(x) G(x-y) h^2(y) \text{d}x \text{d}y \right]^{1/4}.
$$
To see this, we see that,
\begin{equation} \label{eq:chainineq}
\begin{aligned}
\int_{\mathbb{R}^4} F(x) h(x) \text{d}x &= \int_{(\mathbb{R}^4)^2} f(x-e,e) \sqrt{k}(x-e)  \tilde{G}(e) h(x) \text{d}x \text{d}e \\
&= \int_{(\mathbb{R}^4)^2} f(x,e) \sqrt{k}(x) \tilde{G}(e) h(x+e) \text{d}x \text{d}e\\
&\le \left[\int_{(\mathbb{R}^4)^2} f(x,e)^2 \tilde{G}(e) \text{d}x \text{d}e \right]^{1/2} 
\left[\int_{\mathbb{R}^4} k(x) \left(\int h^2(x+e) \tilde{G}(e) \text{d}e\right) \text{d} x\right]^{1/2}\\
& \le \left[\int_{\mathbb{R}^4} k(x)^2 \text{d} x \right]^{1/4}
 \left[\int_{\mathbb{R}^4} \left(\int_{\mathbb{R}^4} h^2(x+e) \tilde{G}(e) \text{d}e\right)^2 \text{d}x  \right]^{1/4}\\
& = \left[\int_{(\mathbb{R}^4)^3} h^2(x+e_1) \tilde{G}(e_1) \tilde{G}(e_2) h^2(x+e_1) \text{d}x \text{d}e_1 \text{d}e_2  \right]^{1/4}\\
& = \left[\int_{(\mathbb{R}^4)^2} h^2(x) G(x-y) h^2(y) \text{d}x \text{d}y \right]^{1/4}.
\end{aligned}
\end{equation}
Then, inequality \eqref{eq:chainineq} shows that
\begin{align*}
A=\langle P_{\tau}F, F \rangle
\le \| P_{\tau}F \|_{G}, 
\end{align*}
where $\| P_{\tau}F \|_G:=[\int_{(\mathbb{R}^4)^2} (P_{\tau}F)^2(x) G(x-y) (P_{\tau}F)^2 (y) dxdy]^{1/4}$ 
and $\langle \cdot \rangle$ is the standard $L^2$ inner product.  
Then, 
\begin{align*}
& A=\langle P_{\tau}F, F \rangle
=\langle P_{\tau}F, (I-2^{-1}\Delta)P_{\tau}F \rangle
\\
&= \| P_{\tau}F \|_G^2 \left\langle \frac{P_{\tau}F}{\| P_{\tau}F \|_G}, 
(I-2^{-1}\Delta)\frac{P_{\tau}F}{\| P_{\tau}F \|_G} \right \rangle
\ge \| P_{\tau}F \|_G^2  c_0 \ge A^2 c_0. 
\end{align*}
By dividing by $A c_0$, we see that $c_0^{-1} \ge A$. As this is true for arbitrary functions $f$ and $k$, this implies that $\rho \le c_0^{-1}$. 
Hence, for any $0< \epsilon <\rho$ there is $f$ such that
\begin{align*}
\frac{1}{\rho-\epsilon}
>\int_{\mathbb{R}^4} f^2(x) dx +\frac{1}{2}\int_{\mathbb{R}^4} |\nabla f(x)|^2 \text{d}x .
\end{align*}

Now, we can return to proving $M(\rho^{-1}) \le 1$. 
If we set $g(x)=f(x)/ \| f\|_2$, 
we have 
\begin{equation} \label{eq:auxtomainopt}
\begin{aligned}
&\frac{1}{\rho-\epsilon}
[ \int_{(\mathbb{R}^4)^2} g^2(x) G (x-y) g^2(y) \text{d}x \text{d}y ]^{1/4}
-\frac{1}{2}\int_{\mathbb{R}^4} |\nabla g(x) |^2 \text{d}x\\
=& \frac{1}{\rho-\epsilon}-\frac{1}{2}\int_{\mathbb{R}^4} |\nabla g(x) |^2 \text{d}x\\
\ge& \left\{ \int_{\mathbb{R}^4} f^2(x) dx +\frac{1}{2}\int_{\mathbb{R}^4} |\nabla f(x)|^2 \text{d}x \right\}
\|f\|_{L^2}^{-2} -\frac{1}{2} \|f\|_{L^2}^{-2} \|\nabla f \|_{L^2}^{2}
=1
\end{aligned}
\end{equation}
for $\int_{(\mathbb{R}^4)^2} g^2(x) G (x-y) g^2(y) \text{d}x \text{d}y=1$. 
This shows that for any $\epsilon$ that $1 \le M\left(\frac{1}{\rho - \epsilon}\right)$. 
Since $\epsilon$ is arbitrary, this implies that $1\le M(\frac{1}{\rho})$ 
and hence $\rho \le \frac{\tilde{\kappa}(4,2)^2}{\sqrt{2}}$. 
\end{proof}

\section{Self-intersection: The Proof of Theorem \ref{p4}}\label{sec:self-inter}

In this section, we will provide the proof of Theorem \ref{p4}. 

\begin{proof}
We set 
\begin{align*}
B(I)=\beta_{t-s}\circ \theta_s 
\end{align*}
and 
\begin{align*}
A(I,J)=\int_I\int_J G(B_s-B_t) \text{d}s\text{d}t 
\end{align*}
such as  \cite[(2.3), (2.4)]{BC04}. 
Note that $B([1/2,1])=_d\beta_{1/2}=_d \beta_1/2 $ and $A([0,1/2]; [1/2,1])=_d 1/2 \int_{0}^1 \int_{0}^1 G(B_t-B'_s) \text{d}t \text{d}s$.  
Moreover, $B([1/2,1])$ is independent of $\beta_{1/2}$. 
Then,  to show the upper bound of \eqref{p4*}, we only have to repeat of the proof of the upper bound of \cite[(3.3)]{BC04}. 

Now we show the lower bound. 
Let
\begin{align*}
C_n
=\sum_{k=1}^{n-1}
A([0,k];[k,k+1])
\end{align*}
for $n=1,2\ldots$ 
We prove
\begin{align}\label{proms}
\liminf_{n\to \infty}
\frac{1}{n} \log \mathbb{E} \exp (\lambda C_n^{1/2})
 \ge \frac{\lambda^2\tilde{k}^2(4,2)}{4}
\end{align}
for $\lambda>0$, 
which corresponds to \cite[(3.9)]{BC04}.
Set $L(t,x)=\int_0^t \tilde{G}(B_s- x) \text{d}s$. 
Then, we have
 \begin{align*}
\left(\iint_{0\le s\le t \le n} G(B_s-B_t) \text{d}s\text{d}t\right)^{1/2}
=& \frac{1}{\sqrt{2}} \left(\int_{\RR^4} L^2(n,x)  \text{d}x\right)^{1/2}\\
\ge& \frac{1}{\sqrt{2}} \int_{\RR^4} f(x) L(n,x)  \text{d}x \\
=& \frac{1}{\sqrt{2}}  \int_0^n \tilde{G}\ast f(B_t)  \text{d}t
\end{align*}
for 
\begin{align}\label{intl2}
\int_{\RR^4} f^2 (x) \text{d}x=1.  
\end{align}
Therefore, by Feynman-Kac formula, 
\begin{align*}
\liminf_{n\to \infty}
&\frac{1}{n} \log \mathbb{E} \exp \left(\lambda \left(\iint_{0\le s\le t \le n} G(B_s-B_t) \text{d}s\text{d}t\right)^{1/2} \right)\\
 \ge & \sup_{g} 
 \bigg\{\frac{\lambda}{\sqrt{2}} \int_{\RR^4} \tilde{G}\ast f(x) g^2(x) \text{d}x -\frac{1}{2}  \int_{\RR^4} |\nabla g(x)|^2 \text{d}x \bigg\}.
\end{align*}
Taking the supremum over $f$ with \eqref{intl2}, it is larger than or equal to
\begin{align*}
 \sup_{g} 
 \bigg\{\frac{\lambda}{\sqrt{2}} \left(\iint_{(\RR^4)^2}  g^2(x)  \tilde{G}(x-y) g^2(y) \text{d}x \text{d}y\right)^{1/2}
 -\frac{1}{2}  \int_{\RR^4} |\nabla g(x)|^2 \text{d}x \bigg\}.
\end{align*}
Therefore, by the same proof as \cite[(3.9)]{BC04}, we obtain \eqref{proms}. 
\end{proof}

\appendix

 \section{Regularizing the singularity near the origin}\label{ref:appendixA}

There are difficulties with dealing with the singularity around the origin when deriving an upper bound for the high moments. In this section, we will consider the moments of the following function,
\begin{equation*}
    \int_{0}^{\tau_1} \int_{0}^{\tau_2} \frac{\mathbbm{1}(|B_t - B'_s| \le \epsilon)}{|B_t - B'_s|^2}  \text{d} t \text{d}s,
\end{equation*}
where $B_t$ and $B'_s$ are independent Brownian motions and $\tau_1,\tau_2$ are independent exponential random variables of rate 1. The new factor here is the introduction of the cutoff $\mathbbm{1}(|B_t - B'_s|)$. 
An expression of the $n$-th moment of this term is given by,
\begin{equation*}
   \mathbb{E}_{\tau}\left[ \int_{[0,\tau_1]^n} \text{d}t_1\ldots \text{d}t_n \int_{[0,\tau_2]^n} \text{d}s_1 \ldots \text{d} s_n \mathbb{E}_{B,B'}\left(\prod_{i=1}^n \frac{\mathbbm{1}(|B_{t_i} - B'_{s_i}| \le \epsilon)}{|B_{t_1} - B'_{s_i}|^2} \right)\right].
\end{equation*}
The first expectation is with respect to the exponential random variable $\tau$. 
The expectation inside is with respect to the Brownian motions $B$ and $B'$.

Let us give some intuition on why this cutoff will give a subleading order term. 
If we are interested in computing the $n$-th moment of the term without cutoff, then the contribution mostly comes when $\tau_1,\tau_2$ is $\approx n$. 
At this scale, the ordered consecutive differences (assume $t_1 \le t_2\ldots \le t_n$ then the consecutive differences would be $t_k-t_{k-1}$) would approximately be of $O(1)$. 
The partial differences $B_{t_i} - B_{t_{i-1}}$ would fluctuate to within $O(1)$ as well. 
Thus, it becomes increasingly unlikely that they could be confined to a neighborhood of size $O(\epsilon)$, as would be needed by the term $\mathbbm{1}(|B_{t_i} - B'_{s_i}|)$. In the remainder of this section, we will try to formalize this intuition.

We start with a lemma that controls some of the expectations that we would see.
\begin{lem}
We have the following estimates. There is some universal constant $C$ not dependent on $\epsilon$ such that,
\begin{equation} \label{eq:intBres}
\begin{aligned}
    & \mathbb{E}_{B}\left[ \frac{\mathbbm{1}(|B_t -x|\le \epsilon)}{|B_t - x|^2} \right] \le  C \min\left( \frac{1}{|x|^2}, \frac{\epsilon^2}{t^2}, \frac{1}{t} \right) \le  \frac{C}{|x|} \min\left( \frac{\epsilon}{t}, \frac{1}{\sqrt{t}}\right),\\
    & \mathbb{E}_{B} \left[ \frac{\mathbbm{1}(|B_t - x| \le \epsilon)}{|B_t - y|^2} \right] \le  C \min\left(\frac{1}{|y|^2}, \frac{\epsilon^2}{t^2}, \frac{1}{t}\right).
\end{aligned}
\end{equation}
\end{lem}
\begin{proof}
In the course of the proof, $C$ is a constant that is allowed to change from line to line. 
We start with considering the expectation of $\frac{\mathbbm{1}(|B_t - x| \le \epsilon)}{|B_t - x|^2}$. 

There are a few cases to consider.
The first case is when $\epsilon$ is larger than $\sqrt{t}$. 
In this case, we may drop the restriction $\mathbbm{1}(|B_t -x| \le \epsilon)$ and use the estimate (4.8) from  \cite[Lemma 4.2]{DO}. 
In the case that $\frac{|x|}{4} \ge \sqrt{t} \ge \epsilon$ , we write the integral as,
\begin{equation*}
\begin{aligned}
    \frac{C}{t^2}\int_{|z-x| \le \epsilon} \frac{\exp[-|z|^2/t]}{|z-x|^2} \text{d} z  &= \frac{C}{t^2} \int_{|z'| \le \epsilon} \frac{\exp[-|z'|^2/t -2 \langle z',x \rangle/t - x^2/t]}{|z'|^2} \text{d} z'\\
    & \le \frac{C \exp[-x^2/t]}{t^2}  \int_{|z'| \le \epsilon} \frac{1}{|z'|^2} \text{d}z' = \frac{C \epsilon^2 \exp[-x^2/(2t)]}{t^2} \le \frac{C }{|x|^2}.
\end{aligned}
\end{equation*}
To get the last line, we used the fact that when $\epsilon \le \frac{|x|}{4}$ we have that $2\langle z',x \rangle \le |x|^2/(2t)$. Furthermore, using that $|x|^2 \ge t$ and that $\sqrt{t} \ge \epsilon$, we can say there is some constant $C$ such that 
$ \frac{C }{|x|^2} \ge \frac{1}{t} \exp[- x^2/(2t)] \ge \frac{\epsilon^2}{t^2} \exp[-x^2/(2t)]$. 
When $\sqrt{t} \ge \epsilon \ge \frac{|x|}{4}$ or $\sqrt{t} \ge \frac{|x|}{4} \ge  \epsilon$, we can bound the integral as follows,
\begin{equation*}
    \frac{C}{t^2} \int_{|z-x| \le \epsilon} \frac{\exp[-|z|^2/t]}{|z-x|^2 } \text{d}z \le \frac{C}{t^2} \int_{|z-x| \le \epsilon} \frac{1}{|z-x|^2} \text{d} z  = \frac{C \epsilon^2}{t^2} \le \frac{C }{t} \le \frac{16 C \epsilon}{|x|^2}.  
\end{equation*}
This gives the first part of the lemma. 
Now, we consider the integral of $\frac{\mathbbm{1}(|B_t - x| \le \epsilon)}{|B_t - y|^2}$. 
When $\epsilon \le |x- y|/4$, then $|B_t - x| \le \epsilon$ will imply $|B_t - y| \le 3|x- y|/4$ and  thus,
\begin{equation*}
    \mathbb{E} \left[ \frac{\mathbbm{1}(|B_t-x| \le \epsilon)}{|B_t-y|^2} \right] \le \frac{4}{|x-y|^2}\mathbbm{E}[\mathbbm{1}(|B_t-x| \le \epsilon)] \le \frac{C}{|x-y|^2} \min\left[\frac{\epsilon^4}{t^2},1\right].
\end{equation*}
If instead $\epsilon \ge |x-y|/4$, then we can say that $\mathbbm{1}(|B_t - x| \le \epsilon) \le \mathbbm{1}(|B_t - y| \le 5 \epsilon)$, and we can then use the estimates on the moments of $\mathbb{E}\left[ \frac{\mathbbm{1}(|B_t - x| \le \epsilon}{|B_t -x |^2}\right]$. 
Since we have that $\epsilon \le |x-y|/4$, we see that 
$$
\frac{C}{|x-y|^2} \frac{\epsilon^4}{t^2} \le \frac{C \epsilon^2}{16 t^2}.
$$
If we knew instead that $\sqrt{t} \le \epsilon$, then we have that,
$$
\frac{C}{|x-y|^2} \le \frac{16C}{\epsilon^2} \le \frac{16C}{t}.
$$

Now, we need to prove that $\frac{1}{|x-y|^2} \mathbb{E}[\mathbbm{1}(|B_t-x|\le \epsilon)] \le \frac{C}{|y|^2}$ for some constant $C$. 
If $|x-y| \ge |y|/4$, we would be done. 
If not, then we have that $|x| \ge |y| -|x-y| \ge 3|y|/4$. Furthermore, $\epsilon \le |x-y|/4 \le |y|/16 \le |x|/12$. 
Thus, we can write,
$$
\mathbb{E}[\mathbbm{1}(|B_t - x| \le \epsilon)] = \frac{C}{t^2} \int_{|z| \le \epsilon} \exp[-|z|^2/t - 2 \langle z,x \rangle/t - |x|^2/t] \text{d}z.
$$
We have that $ - 2 \langle z, x \rangle /t - |x|^2/2 \le - |x|^2/(2t)$ since $|z| \le \epsilon \le |x|/12$. We thus have,
$$
\begin{aligned}
\frac{C}{t^2} \int_{|z| \le \epsilon} \exp[-|z|^2/t - 2 \langle z,x \rangle/t - |x|^2/t] \text{d} z
&\le \frac{C}{t^2}\exp[-|x|^2/(2t)] \int_{|z| \le \epsilon} \exp[-|z|^2/t] \text{d} z\\
& \le C\min(1, \frac{\epsilon^4}{t^2}) \exp[-|x|^2/(2t)].
\end{aligned}
$$
If $\epsilon^2 \le t$, we use the fact that $\epsilon^ 2\le |x-y|^2/16 $ and derive that
$$
\frac{\epsilon^4}{|x-y|^2t^2} \exp[-|x|^2/(2t)] \le \frac{1}{16 t} \exp[-|x|^2/(2t)] \le \frac{C}{|x|^2} \le \frac{ C}{9 |y|^2},
$$
in addition if $\epsilon^2 \ge t$, we instead get that,
$$
\frac{1}{|x-y|^2}\exp[-|x|^2/(2t)] \le \frac{16}{\epsilon^2} \exp[-|x|^2/(2t)] \le \frac{16}{t} \exp[-|x|^2/(2t)] \le \frac{16C}{|x|^2} \le \frac{256 C}{9|y|^2}. 
$$
Then, we obtain the desired result. 
\end{proof}

As a consequence of these estimates, we can derive the following estimates,
\begin{lem}
There is a universal constant $C$ not dependent on $\epsilon$ such that the following estimates hold: 
\begin{equation}\label{eq:intBres2}
\begin{aligned}
    &\mathbb{E}_{B}\left[ \frac{\mathbbm{1}(|B_t - x| \le \epsilon)}{|B_t -x||B_t - y|} \right] \le C\min\left(\frac{1}{|x|}, \frac{1}{|y|}\right) \min\left(\frac{\epsilon}{t},\frac{1}{\sqrt{t}}\right),\\
    & \mathbb{E}_{B}\left[ \frac{\mathbbm{1}(|B_t-x| \le \epsilon)}{|B_t-x|^2|B_t - y|} \right] \le C \frac{1}{|x||x-y|} \min\left(\frac{\epsilon}{t}, \frac{1}{\sqrt{t}}\right).
\end{aligned}
\end{equation}
\begin{proof}
The first inequality can be derived via the Cauchy-Schwarz inequality. 
That is, we have,
\begin{equation*}
    \mathbb{E}\left[ \frac{\mathbbm{1}(|B_t - x| \le \epsilon)}{|B_t -x||B_t - y|} \right] \le \mathbb{E}\left[ \frac{\mathbbm{1}(|B_t - x| \le \epsilon)}{|B_t -x|^2} \right]^{1/2} \mathbb{E}\left[ \frac{\mathbbm{1}(|B_t - x| \le \epsilon)}{|B_t -y|^2} \right]^{1/2}.
\end{equation*}
If $\frac{1}{|x|} \le \frac{1}{|y|}$, we can bound $\mathbb{E}\left[ \frac{\mathbbm{1}(|B_t - x| \le \epsilon)}{|B_t -x|^2} \right]^{1/2}$ by $\frac{1}{|x|^2}$, using the first inequality of \eqref{eq:intBres}.  the integral $\mathbb{E}\left[ \frac{\mathbbm{1}(|B_t - x| \le \epsilon)}{|B_t -y|^2}\right]$ can be bounded by $\min\left(\frac{\epsilon^2}{t^2}, \frac{1}{\sqrt{t}}\right)$ by the second inequality of \eqref{eq:intBres}. 
If instead $\frac{1}{|y|} \le \frac{1}{|x|}$, we can go the other way around. 
To deal with the second inequality of \eqref{eq:intBres2}, we instead use,
\begin{equation*}
\begin{aligned}
    \mathbb{E}_{B}\left[ \frac{\mathbbm{1}(|B_t-x| \le \epsilon)}{|B_t-x|^2|B_t - y|} \right]  &\le \frac{1}{|x-y|} \mathbb{E}\left[ \frac{\mathbbm{1}(|B_t-x| \le \epsilon)}{|B_t-x||B_t - y|} \right] + \frac{1}{|x-y|}\mathbb{E}\left[ \frac{\mathbbm{1}(|B_t-x| \le \epsilon)}{|B_t-x|^2} \right]\\
    &\le \frac{C}{|x-y||x|} \min\left(\frac{\epsilon}{t}, \frac{1}{\sqrt{t}}\right).
\end{aligned}
\end{equation*}
The last line used the first inequality of \eqref{eq:intBres} and the first inequality of \eqref{eq:intBres2}. 
\end{proof}

The main improvement in this lemma compared to \cite[Lemma 4.2]{DO} is the change of the time bound to $\min\left(\frac{\epsilon}{t}, \frac{1}{\sqrt{t}}\right)$. We are now in good shape to bound the moments of $\int_{0}^{\tau_1} \int_{0}^{\tau_2} \frac{\mathbbm{1}(|B_t - B_s'| \le \epsilon)}{|B_t - B_s'|^2} \text{d} t \text{d}s$.

\begin{lem} \label{lem:Brownianorigin}
There is some universal constant $C$, not dependent on $\epsilon$ or $n$, such that we have the following moment estimates: 
\begin{equation*}
    \mathbb{E}\left(\int_{0}^{\tau_1} \int_0^{\tau_2} \frac{\mathbbm{1}(|B_t - B'_s| \le \epsilon)}{|B_t - B'_s|^2} \text{d}t \text{d}s\right)^n \le C^n \sqrt{\epsilon}^n (n!)^2.
\end{equation*}
Here, $\tau_1,\tau_2$ are two independent exponential random variables with rate 1 and $B$, $B'$ are two independent Brownian motions.
\end{lem}
\begin{proof}

We start with bounding the more general quantity,
\begin{equation*}
\begin{aligned}
    &\mathbb{E}_B  \prod_{i=1}^{n} \frac{\mathbbm{1}(|B_{t_i} - y_i| \le \epsilon)}{|B_{t_i} - y_i|^2}\\
     &=\mathbb{E}_B \prod_{i=1}^{n-1} \frac{\mathbbm{1}(|B_{t_i} - y_i| \le \epsilon)}{|B_{t_i} - y_i|^2} \frac{\mathbbm{1}(|B_{t_n}- B_{t_{n-1}} - (y_n - B_{t_{n-1}})| \le \epsilon)}{|(B_{t_n} - B_{t_{n-1}}) - (y_n - B_{t_{n-1}})|^2}\\
     &\le C  \mathbb{E}_B   \prod_{i=1}^{n-1} \frac{\mathbbm{1}(|B_{t_i} - y_i| \le \epsilon)}{|B_{t_i} - y_i|^2} \frac{1}{|B_{t_{n-1}} -y_n|} \min\left( \frac{\epsilon}{t_n - t_{n-1}}, \frac{1}{\sqrt{t_n - t_{n-1}}}\right).
\end{aligned}
\end{equation*}
To get the last inequality, we used the fact that the difference $B_{t_n}- B_{t_{n-1}}$ is independent of the Brownian walk up to time $t_{n-1}$ and is distributed according to a Brownian motion at time $t_n - t_{n-1}$. 
We then used the first inequality of \eqref{eq:intBres}. 
At this point, we can proceed in an inductive fashion. We have,
\begin{equation*}
\begin{aligned}
     & \mathbb{E}_B   \prod_{i=1}^{n-1} \frac{\mathbbm{1}(|B_{t_i} - y_i| \le \epsilon)}{|B_{t_i} - y_i|^2} \frac{1}{|B_{t_{n-1}} -y_n|} \min\left( \frac{\epsilon}{t_n - t_{n-1}}, \frac{1}{\sqrt{t_n - t_{n-1}}}\right)\\
     &=\mathbb{E}_B   \prod_{i=1}^{n-2} \frac{\mathbbm{1}(|B_{t_i} - y_i| \le \epsilon)}{|B_{t_i} - y_i|^2} \frac{\mathbbm{1}(|(B_{t_{n-1}} -B_{t_{n-2}}) - (y_{n-1} -B_{t_{n-2}}) | \le \epsilon)}{|(B_{t_{n-1}} -B_{t_{n-2}}) - (y_{n-1} -B_{t_{n-2}}) |^2|(B_{t_{n-1}} -B_{t_{n-2}}) - (y_{n} -B_{t_{n-2}}) |  }\\
     &\le C \mathbb{E}_B \prod_{i=1}^{n-2} \frac{\mathbbm{1}(|B_{t_i} - y_i| \le \epsilon)}{|B_{t_i} - y_i|^2} \frac{1}{|B_{t_{n-1}} -y_n||y_n - y_{n-1}|} \min\left(\frac{\epsilon}{t_{n-1} - t_{n-2}}, \frac{1}{\sqrt{t_{n-1} - t_{n-2}}}\right). 
\end{aligned}
\end{equation*}
Combining these steps we see that,
\begin{equation} \label{eq:iterres}
    \mathbb{E}_B  \prod_{i=1}^{n} \frac{\mathbbm{1}(|B_{t_i} - y_i| \le \epsilon)}{|B_{t_i} - y_i|^2} \le C^n \prod_{i=1}^n \frac{1}{|y_i - y_{i-1}|} \min\left(\frac{\epsilon}{t_{i-1} - t_{i-2}}, \frac{1}{\sqrt{t_{i-1} - t_{i-2}}}\right).
\end{equation}
The $n$-th moment of $\int_{0}^{\tau_1} \int_0^{\tau_2} \frac{\mathbbm{1}(|B_t - B'_s| \le \epsilon)}{|B_t - B'_s|^2} \text{d}t \text{d}s$ can be expressed as,
\begin{equation} \label{eq:nthmomentres}
\begin{aligned}
    &n!\mathbb{E}_{\tau} \int_{0 \le t_1\le t_2\ldots \le t_n \le \tau_1} \text{d}t_1\ldots \text{d}t_n  \int_{[0,\tau_2]^n} \text{d} s_1 \ldots \text{d}s_n \mathbb{E}_{B,B'}\left[\prod_{i=1}^n \frac{\mathbbm{1}(|B_{t_i} - B'_{s_i}| \le \epsilon)}{|B_{t_i} - B'_{s_i}|^2}\right]\\
    &\le C^n n!\mathbb{E}_{\tau} \int_{0 \le t_1\le t_2\ldots \le t_n \le \tau_1} \text{d}t_1\ldots \text{d}t_n \min\left( \frac{\epsilon}{t_{i} -t_{i-1}}, \frac{1}{\sqrt{t_i - t_{i-1}}}\right)  \\
    &\times \int_{[0,\tau_2]^n} \text{d} s_1 \ldots \text{d}s_n \prod_{i=1}^n  \mathbb{E}_{B'}\left[\prod_{i=1}^n \frac{1}{|B'_{s_{i}} - B'_{s_{i-1}}|} \right]\\
    & \le n! \mathbb{E}_{\tau} C^n \int_{0 \le t_1\le t_2\ldots \le t_n \le \tau_1} \text{d}t_1\ldots \text{d}t_n \min\left( \frac{\epsilon}{t_{i} -t_{i-1}}, \frac{1}{\sqrt{t_i - t_{i-1}}}\right) \\& \times n! C^n \int_{0 \le s_1 \le s_2\ldots \le s_n \le \tau_2 } \prod_{i=1}^n \frac{1}{\sqrt{s_i - s_{i-1}}}. 
\end{aligned}
\end{equation}
To obtain the second inequality, we used \eqref{eq:iterres}. 

To get the final inequality, we used equation \cite[(4.18)]{DO}.
By scaling one has
\begin{align*}
&\int_{0 \le s_1 \le s_2\ldots \le s_n \le \tau_2 } \text{d}s_1\ldots \text{d}s_n \prod_{i=1}^n \frac{1}{\sqrt{s_i - s_{i-1}}} \\
 = & \tau_2^{n/2}\int_{0 \le s_1 \le s_2\ldots \le s_n \le 1}  \text{d}s_1\ldots \text{d}s_n \prod_{i=1}^n \frac{1}{\sqrt{s_i - s_{i-1}}}  
 \le \frac{C^n \tau_2^{n/2}}{(n!)^{1/2}}.
\end{align*}
One can see equation \cite[(4.21)]{DO} for a reference. 
The more important term to deal with is,
\begin{equation*}
    \int_{0 \le t_1\le t_2\ldots \le t_n \le \tau_1} \text{d}t_1\ldots \text{d}t_n \min\left( \frac{\epsilon}{t_{i} -t_{i-1}}, \frac{1}{\sqrt{t_i - t_{i-1}}}\right)
\end{equation*}
 Let $I_k$ denote the value of the integral,
 \begin{equation*}
     I_k:= \int_{0 \le \theta_1 \le \theta_2\ldots \le \theta_k \le 1} \text{d}\theta_1 \ldots \text{d}\theta_k  \prod_{i=1}^k \frac{1}{\sqrt{\theta_i - \theta_{i-1}}}.
 \end{equation*}
 Notice that $I_k$ satisfies the relation,
 \begin{equation*}
     I_k = \int_0^1 \frac{I_{k-1} (1- \theta_1)^{(k-1)/2}}{\sqrt{\theta_1}} \text{d}\theta_1,
 \end{equation*}
and $I_0$ is also understood to be $0$. 
 By induction, we will prove the following inequality,
 \begin{equation*}
     \int_{0 \le t_1 \le t_2 \ldots \le t_n \le \tau_1} \text{d}t_1\ldots \text{d}t_n
      \min\left(\frac{\epsilon}{t_i - t_{i-1}}, \frac{1}{\sqrt{t_i -t_{i-1}}}\right) \le 2^n \sqrt{\epsilon}^n \sum_{k=0}^n (\tau_1)^{k/2} I_{k} {n \choose k}.
 \end{equation*}
 The base case $n=1$ can be bounded from above by the integral
 \begin{align*}
     \int_{0}^{\tau_1} \min\left( \frac{1}{\sqrt{t}}, \frac{\epsilon}{t}\right) \text{d} t 
     \le& \int_{0}^{\epsilon} \frac{1}{\sqrt{t}} \text{d} t + \int_{\epsilon}^{\tau_1} \frac{\epsilon}{t} \text{d} t \\
     \le & 2\sqrt{\epsilon} + \sqrt{\epsilon}  \int_{\epsilon}^{\tau_1} \frac{1}{\sqrt{t}} \text{d}t \le 2 \sqrt{\epsilon} + \sqrt{\epsilon} \int_{0}^{\tau_1} \frac{1}{\sqrt{t}} \text{d}t = 2 \sqrt{\epsilon}I_0 + \sqrt{\epsilon} \sqrt{\tau_1} I_1.
 \end{align*}
The second term integral above is understood to be $0$ if $\tau$ is less than $\epsilon$. $\sqrt{\epsilon}$ will still be an upper bound if $\tau$ is less than $\epsilon$.

Now, we can proceed with our induction. 
We have, 
\begin{equation*}
\begin{aligned}
    &\int_{0 \le t_1 \le t_2 \ldots \le t_n \le \tau_1} \text{d}t_1\ldots \text{d}t_n
    \min\left(\frac{\epsilon}{t_i - t_{i-1}}, \frac{1}{\sqrt{t_i -t_{i-1}}}\right)\\ 
    & = \int_{0}^{\tau_1} \text{d}t_1 \min\left(\frac{1}{\sqrt{t_1}}, \frac{\epsilon}{t_1}\right) 
    \int_{t_1 \le t_2 \le \ldots \le t_n \le \tau_1} \text{d}t_2 \ldots \text{d}t_n  \prod_{i=2}^n \min\left(\frac{1}{\sqrt{t_i- t_{i-1}}}, \frac{\epsilon}{t_{i} -t_{i-1}}\right)  \\
    & \le  \int_{0}^{\tau_1} \text{d}t_1 \min\left(\frac{1}{\sqrt{t_1}}, \frac{\epsilon}{t_1} \right) 
    \int_{0 \le t_1' \le \ldots \le t_n'\le \tau_1- t_1} \text{d}t_2' \ldots \text{d}t_n'  \prod_{i=1}^{n-1}\min\left(\frac{1}{\sqrt{t'_i- t'_{i-1}}}, \frac{\epsilon}{t'_{i} -t'_{i-1}}\right)\\
    & \le \int_{0}^{\tau_1} \text{d}t_1 \min\left(\frac{1}{\sqrt{t_1}}, \frac{\epsilon}{t_1} \right) 2^{n-1} \sqrt{\epsilon}^{n-1} \sum_{k=0}^{n-1} (\tau_1 -t_1)^{k/2} I_k {n-1 \choose k}.
\end{aligned}  
\end{equation*}
and
\begin{equation*}
\begin{aligned}
    & \le \int_{0}^{\epsilon} \frac{1}{\sqrt{t_1}} 2^{n-1} \sqrt{\epsilon}^{n-1} \sum_{k=0}^{n-1} (\tau_1)^{k/2} {n-1 \choose k} \text{d}t_1 + \int_{\epsilon}^{\tau_1} \frac{\epsilon}{t_1} 2^{n-1} \sqrt{\epsilon}^{n-1} \sum_{k=0}^{n-1} (\tau_1 -t_1)^{k/2} I_k {n-1 \choose k} \text{d} t_1\\
    & \le 2^{n} \sqrt{\epsilon}^{n} \sum_{k=0}^{n-1} (\tau_1)^{k/2} {n-1 \choose k} + \sqrt{\epsilon} \int_{0}^{\tau_1} \frac{1}{\sqrt{t_1}} 2^{n-1} \sqrt{\epsilon}^{n-1} \sum_{k=0}^{n-1} (\tau_1 -t_1)^{k/2} I_k {n-1 \choose k} \text{d} t_1\\
    & \le 2^n \sqrt{\epsilon}^{n} \sum_{k=0}^{n} I_k (\tau_1)^{k/2} {n \choose k}.
\end{aligned}
\end{equation*}
Thus, our expression for \eqref{eq:nthmomentres} can be bounded by
\begin{equation*}
\begin{aligned}
     C^n(n!)^2  2^n (\sqrt{\epsilon})^n\mathbb{E}_{\tau} \sum_{k=0}^n{n \choose k} I_k I_n \tau_1^{k/2} \tau_2^{n/2} 
    &\le C^n (n!)^2 (\sqrt{\epsilon})^n \sum_{k=0}^n {n \choose k} \frac{1}{\sqrt{k!}} \frac{1}{\sqrt{n!}} \left(\frac{n}{2}\right)! \left(\frac{k}{2}\right)!\\
    & \le  C^n (n!)^2 (\sqrt{\epsilon})^n \sum_{k=0}^n {n \choose k} \le C^n (n!)^2 (\sqrt{\epsilon})^n. 
\end{aligned}
\end{equation*}
Therefore, we obtain the desired result. 
\end{proof}

\end{lem}

\section{Analysis of the Constrained Optimization Problem} \label{sec:AppendixB}

Fix a function $M$ that is bounded and with finite support. 
In this section, we will analyze optimization problems of the following form,
\begin{equation*}
\begin{aligned}
    & K_{\epsilon}:=\{ k: \epsilon^4 \sum_{z \in \epsilon \mathbb{Z}^4} k^2(z) = 1\},\\
    & F_{\epsilon,M}:=\{f:\epsilon^4 \sum_{z \in \epsilon \mathbb{Z}^4} \int_{\mathbb{R}^4} f^2(z,e) M(e)  \text{d}e = 1\} 
\end{aligned}
\end{equation*}
and
\begin{equation*}
\begin{aligned}
    &O_{x_1,x_2,\epsilon,M} 
    := \sup_{k \in K_{\epsilon}}\sup_{ f_1,f_2 \in F_{\epsilon,M}} \epsilon^8 \sum_{z_1 \in \epsilon \mathbb{Z}^4,z_2 \in \epsilon \mathbb{Z}^4}\\
    &\times \int_{(\mathbb{R}^4)^2} \sqrt{k}(z_1) f_1(z_1,e_1) M(e_1) P_{\tau}(z_1 + x_1+e_1 - z_2 - x_2-e_2) M(e_2) \sqrt{k}(z_2) f_2(z_2,e_2) \text{d}e_1 \text{d}e_2.
\end{aligned}
\end{equation*}
Note that $O_{x_1,x_2,\epsilon,M}$ can be understood as an upper bound for the norm of all operators of the following operator on $F_{\epsilon,M}$: 
\begin{equation*}
    (T_{x_1,x_2,\epsilon,M,k} f)(z,e) 
    = \sqrt{k}(z) \epsilon^4 \sum_{\tilde{z} \in \epsilon \mathbb{Z}^4} \int_{\mathbb{R}^4} 
    P_{\tau}(z+ x_1 + e - \tilde{z} - x_2 - \tilde{e})  M(\tilde{e}) \sqrt{k}(\tilde{z}) f(\tilde{z},\tilde{e}) \text{d} \tilde{e} . 
\end{equation*}
We thus see that,
$$
O_{x_1,x_2,\epsilon,M}= \sup_{k \in K_{\epsilon}}\sup_{f_1,f_2 \in F_{\epsilon,M}} \langle f_1, T_{x_2} f_2 \rangle.
$$
Here, $x_1$ and $x_2$ are two points found in $[-\epsilon,\epsilon]^4$. 
The continuous analogue of these quantities can be expressed as follows: 
\begin{equation*}
\begin{aligned}
     & K:=\{ k : \int_{\mathbb{R}^4} \text{d} z k^2(z) = 1\},\\
     & F_{M}:= \{ f: \int_{(\mathbb{R}^4)^2} \text{d} z \text{d} e f^2(z,e) M(e) \}
\end{aligned}
\end{equation*}
and
\begin{equation*}
\begin{aligned}
    O_{M} &:= \sup_{k \in K,f_1,f_2 \in F_M} \int_{(\mathbb{R}^4)^4} \sqrt{k}(z_1) f_1(z_1,e_1) M(e_1) P_{\tau}(z_1 + e_1 -z_2 - e_2)\\
    & \times M(e_2) \sqrt{k}(z_2) f(z_2,e_2) \text{d}z_1 \text{d}z_2 \text{d}e_1 \text{d}e_2.
\end{aligned}
\end{equation*}
We remark that this $O_M$ corresponds to the symmetric operator on $F_M$ given by,
\begin{equation*}
    (T_{M,k} f)(z,e) = \sqrt{k}(z) \int_{(\mathbb{R}^4)^2} P_{\tau}(z + e - \tilde{z}-\tilde{e}) \sqrt{k}(\tilde{z}) M(\tilde{e}) \text{d}\tilde{z} \text{d} \tilde{e}.
\end{equation*}
Thus, $O_M$ would be the same whether we took the maximum over $f_1$,$f_2$ arbitrary or $f_1=f_2$. Let $C$ be a continuous function such that $C(z) \le P_{\tau}(z) \le C(z) + P_{\tau}(z)\mathbbm{1}(|z| \le \delta)$.  
We let the quantities $O_M^{C}, O_{x_1,x_2,\epsilon,M}^{C}$ or $O_{M}^{\delta}$, $O_{x_1,x_2,\epsilon,M}^{\delta}$ denote the analogues of $O_{x_1,x_2,\epsilon,M}$ or $O_{M}$ with the function $P_{\tau}$ in the definition replaced either by $C$ or by $P_{\tau}(z) \mathbbm{1}(|z| \le \delta).$
Clearly we have that,
\begin{equation*}
    O^C_M \le O_M \le O_M^{C} + O_M^{\delta}, O_{x_1,x_2,\epsilon,M}^{C} \le O_{x_1,x_2,\epsilon,M} \le O_{x_1,x_2,\epsilon,M}^{C} + O_{x_1,x_2,\epsilon,M}^{\delta}.
\end{equation*}
We will first argue that independent of $x_1,x_2$ and $\epsilon$, that $O_{x_1,x_2,\epsilon,M}^{\delta}$ will go to $0$ as $\delta$ goes to $0$.

One way to rewrite our maximization problems is as follows. We can consider the normalization
$$
F(z) = \sqrt{\int_{\mathbb{R}^4} f^2(z,e) M(e) \text{d} e}, 
\quad
N_z(e) = \frac{f(z,e)}{F(z)}.
$$
We remark that $\int_{\mathbb{R}^4} F(z)^2 \text{d} z =1$ and for each $z$ we have $\int_{\mathbb{R}^4} (N_z(e))^2 \text{d}e$.

Next we see another way to write the integral expression that appears in $O_M$ (we have similar expressions for $O_{x_1,x_2,\epsilon,M}$) is as follows:
\begin{equation} \label{eq:Orewrite}
    \int_{(\mathbb{R}^4)^2} \text{d}z \text{d}\tilde{z} \sqrt{k}(z) F_1(z) F_2(\tilde{z}) \sqrt{k}(\tilde{z}) \int_{(\mathbb{R}^4)^2} \text{d} e  \text{d} \tilde{e} M(e) (N_z)_1(e) P_{\tau}(z+e -\tilde{z} - \tilde{e}) M(\tilde{e}) (N_z)_2(\tilde{e})
\end{equation}
We need a series of lemmas that analyze this expression based on understanding the value of the integral in the interior. We start with a lemma that derives a bound on what we will consider the `canonical' version of the problem. 

\begin{lem} \label{lem:canonprob}
Consider the following problem: 
\begin{equation} \label{eq:deffrakI}
    \mathfrak{I}:= \sup_{\substack{N_1,N_2:\\\int_{[-1,1]^4} N_i(e)^2 \text{d}e =1}} \int_{([-1,1]^4)^2} N_1(e_1)\frac{1}{|e_1 - e_2|^2} N_2(e_2) \text{d}e_1 \text{d}e_2. 
\end{equation}
Then $\mathfrak{I}$ is bounded. 

\end{lem}
\begin{proof}
Assume for contradiction that $\mathfrak{I}$ is not bounded. Then we can find a sequence of functions $N^{B}_1$ and $N^B_2$ supported on $[-1,1]^4$ and $L^2$ norm 1 such that 
\begin{align} \label{def:NB}
    &\int_{([-1,1]^4)^2} N^B_1(e_1) P_{\tau}(e_1 - e_2) N^B_2(e_2) \text{d}e_1 \text{d} e_2\\
    \notag
     &\ge c \int_{([-1,1]^4)^2} N^B_1(e_1) \frac{1}{|e_1 - e_2|^2} N^B(e_2) \text{d}e_1 \text{d} e_2 \ge B,
\end{align}
where $c$ is a constant so that $P_{\tau}(e_1 -e_2) \ge c \frac{1}{|e_1 - e_2|^2}$, when $e_1,e_2$ are supported in $[-1,1]^4$.

Now consider the large deviation of the following quantity. Let $I$ denote the indicator function $I(x):= \mathbbm{1}(x \in [-1,1]^4)$ and 
\begin{equation*}
   \mathcal{I}:= \int_{0}^{\tau_1} \int_{0}^{\tau_2} (I*I)(B_t - B'_s) \text{d}t \text{d}s.
\end{equation*}

Since the convolution $I*I$ is bounded, we can see that $\frac{1}{n} \log \frac{1}{(n!)^2} \mathbb{E}[(\mathcal{I})^n] < \infty$. 
However, similar to the proof given in Section \ref{sec:lwrbndbrown}, we can prove that
\begin{equation} \label{eq:lwerbndI}
\begin{aligned}
& \liminf_{n\to \infty}\frac{1}{n} \log \frac{1}{(n!)^2} \mathbb{E}[(\mathcal{I})^n] \\
&\ge \sup_{\substack{
\int_{\mathbb{R}^4} k^2(z) \text{d} z =1\\
\int_{(\mathbb{R}^4)^2} f_i^2(z,e) I(e) \text{d}z \text{d} e=1} } \int_{(\mathbb{R}^4)^4} \sqrt{k(z_1)} f_1(z_1,e_1) I(e_1) P_{\tau}(z_1 +e_1 - z_2 - e_2)\\
&  \hspace{3 cm} \times I(e_2) f_2(z_2,e_2) \sqrt{k(z_2)} \text{d}z_1 \text{d}z_2 \text{d}e_1 \text{d}e_2.
\end{aligned}
\end{equation}
Now, we choose $f_i(z_i,e_i)$ of the following form.  If $|z_i| \le \frac{1}{2}$, then we write $f_i(z_i,e_i)$ as $F(z_i) N^B_i(e_i + z_i)$, where $F$ supported on $\left[-\frac{1}{2}, \frac{1}{2}\right]^4$ is a function with norm $1$, namely, $\int_{\mathbb{R}^4} F^2(z) \text{d}z =1$ and $N^B_i$ is the function from \eqref{def:NB}.  One can manifestly see that $\int_{(\mathbb{R}^4)^2} F^2(z) (N^B_i(e))^2 I(e) \text{d} e=1$ by definition. We also fix $k$ to be some function with $L^2$ norm 1. 
With this choice of $k$ and $f_1,f_2$, we see that 
\begin{equation*}
\begin{aligned}
    &\int_{(\mathbb{R}^4)^2} \text{d}z_1 \text{d}z_2 F(z_1) \sqrt{k}(z_1) F(z_2) \sqrt{k}(z_2)\\
    & \times \int_{(\mathbb{R}^4)^2} \text{d}e_1 \text{d}e_2 N_1^B(e_1 +z_1) I(e_1) P_{\tau}(e_1 + z_1 - e_2 - z_2) N^B_2(e_2  +z_2) I(e_2)\\
    & \ge B \int_{(\mathbb{R}^4)^2} \text{d}z_1 \text{d}z_2 F(z_1) \sqrt{k}(z_1) F(z_2) \sqrt{k}(z_2).
\end{aligned}
\end{equation*}
The last inequality merely follows from the condition on $N_i^B$ for \eqref{def:NB}, 
once you use the observation that $|e_1 +z_1|$ will be $<2$ when $|e_1| <1$ and $|z_1|<\frac{1}{2}$. Thus, in the domain of relevance, $I(e_i)$ will just be $1$.  
We can freely take $B$ to $\infty$ while keeping $F$ and $k$ fixed. Thus, the supremum in \eqref{eq:lwerbndI} will be $\infty$. This contradicts that the fact that said supremum should be finite. Thus, it must be the case that the quantity $I$, of interest, must be finite. 

\end{proof}

We can now proceed to relate the more general problem to a bound on the canonical problem. 

\begin{lem}
Assume that $M$ has support $[-S,S]^4$ and that $M$ is bounded from above  by $B$ and from below by $b$ on this support. Then we have the following estimates:
If $|z_1 -z_2| \ge   4 \sqrt{d} |S|$, we have that
\begin{equation*}
\begin{aligned}
    &\sup_{\substack{N_1,N_2: \\ \int_{\mathbb{R}^4} N_i^2(e) M(e) \text{d} e =1}}
\int_{(\mathbb{R}^4)^2} N_1(e_1) M(e) P_{\tau}(z_1 + e_1 - z_2 - e_2) N_2(e_2) M(e_2) \text{d}e_1 \text{d}e_2 \\ 
    & \le  P_{\tau}(\frac{z_1 - z_2}{2}) \sqrt{B[2S]^4}.
\end{aligned}
\end{equation*}

If instead, we assume that $|z_1 -z_2| \le 4\sqrt{d} S$, we have that, 

\begin{equation} \label{eq:z1z2close}
    \begin{aligned}
    &\sup_{\substack{N_1,N_2: \\ \int_{\mathbb{R}^4} N_i^2(e) M(e) \text{d} e =1}} \int_{(\mathbb{R}^4)^2} N_1(e_1) M(e) P_{\tau}(z_1 + e_1 - z_2 - e_2) N_2(e_2) M(e_2) \text{d}e_1 \text{d}e_2   \le\frac{B^2}{b(3S)^2} \mathfrak{I}. 
\end{aligned}
\end{equation}
$\mathfrak{I}$ is the quantity from \eqref{eq:deffrakI}.  
\end{lem}
\begin{proof}
Let us consider the case that $|z_1 - z_2| \ge 4 \sqrt{d}|S| $.  
In this case, we can assert that for any $e_1,e_2$ in the support of $S$, we have that $|z_1 - z_2 + e_1 -e_2| \ge \frac{|z_1 - z_2|}{2}$. $P_{\tau}$ depends only of the norm of its input and is monotone decreasing in its input, thus, $P_{\tau}(z_1 - z_2 + e_1 - e_2) \le P_{\tau}(\frac{z_1-z_2}{2})$ when $e_1$ and $e_2$ are in the support of $M$.

Secondly, we also know that,
\begin{equation*}
    \int_{\mathbb{R}^4 }N_1(e) M(e) \text{d}e \le [\int_{\mathbb{R}^4} (N_1(e))^2 M(e) \text{d}e]^{1/2} [\int_{\mathbb{R}^4} M(e) \text{d} e]^{1/2} \le \sqrt{B[2S]^4}.
\end{equation*}
Because $M$ is bounded from above by $B$, we can then assert that,
 \begin{equation*}
 \begin{aligned}
     & \sup_{\substack{N_1,N_2:\\ \int_{\mathbb{R}^4} N_i^2(e) M(e) \text{d} e =1}} \int_{(\mathbb{R}^4)^2} N_1(e_1) M(e_1) P_{\tau}(z_1 + e_1 - z_2 - e_2) M(e_2) N_2(e_2) \text{d}e_1
     \text{d}e_2 \\
     & \le P_{\tau}(\frac{z_1-z_2}{2})\sup_{\substack{N_1,N_2:\\ \int_{\mathbb{R}^4} N_i^2(e) M(e) \text{d} e =1}} \int_{\mathbb{R}^4} N_1(e_1) M(e_1) \text{d}e_1 \int_{\mathbb{R^4}} N_2(e_2) M(e_2) \text{d} e_2\\
     & \le P_{\tau}(\frac{z_1-z_2}{2}) \sqrt{B[2S]^4}.
\end{aligned}
 \end{equation*}
 If instead, we assume that $|z_1 -z_2| \le 4 \sqrt{d}|S|$, then we instead know that $|z_1 +z_2 - e_1 - e_2| \le 6 \sqrt{d}|S|$ and that 
 \begin{equation*}
    \begin{aligned}
     & \sup_{\substack{N_1,N_2:\\ \int_{\mathbb{R}^4} N_i^2(e) M(e) \text{d} e =1}} \int_{(\mathbb{R}^4)^2} N_1(e_1) M(e_1)\times P_{\tau}(z_1 + e_1 - z_2 - e_2) M(e_2) N_2(e_2) \text{d}e_1
     \text{d}e_2\\
     & = \sup_{\substack{N_1,N_2:\\ \int_{\mathbb{R}^4} N_1^2(e-z_1 +z_2) M(e-z_1 + z_2) \text{d} e =1\\
     \int_{\mathbb{R}^4} N_2^2(e) M(e) \text{d} e =1}} \int_{(\mathbb{R}^4)^2} N_1(e_1 - z_1 +z_2) M(e_1- z_1 + z_2)\\
     & \hspace{3 cm} \times P_{\tau}(e_1 -e_2) M(e_2) N_2(e_2) \text{d}e_1
     \text{d}e_2\\
     & \le B^2 \sup_{\substack{\tilde{N}_1,\tilde{N}_2:\\  b \int_{[-3S, 3S]^4} \tilde{N}_i^2(e) \text{d} e =1}} \int_{([-3S, 3S]^4)^2} \text{d}e_1 \text{d}e_2 \tilde{N}_1(e_1) \times P_{\tau}(e_1 - e_2) \tilde{N}_2(e_2)\\
     &=  \frac{B^2(3S)^2}{b} \sup_{\substack{\tilde{N}_1,\tilde{N}_2:\\ \int_{[-1,1]^4}\tilde{N}_1^2(e)\text{d}e =1  }} \int_{[-1,1]^4} \text{d}e_1 \text{d}e_2 \tilde{N}_1(e_1) \frac{1}{|e_1- e_2|^2} \tilde{N}_2(e_2). 
     \end{aligned}
 \end{equation*}
 To get the second line above, we changed variable of $e_1 \to e_1 -z_1 + z_2$. 
 We need to make more observations to derive the second to last line. Firstly, one can use the assumption that $M$ is bounded by $B$. This gives the factor of $B^2$ outside. 
 Secondly, the support of the shifted functions $N_1(e_1 - z_1 +z_2)$ must be restricted to the domain $[-3S,3S]^4$. 
 The support of $N_2(e_2)$ must also be restricted to this domain. Thus, we can increase the domain of integration to $[-3S,3S]^4$. Finally, since $M$ is bounded below by $b$, we would also know that $b\int_{[-S,S]^4} N_1^2(e) \text{d} e \le 1 $. By assumption, the functions $N_1$ would also have support restricted to $[-S,S]^4$; thus, $b\int_{[-3S,3S]^4} N_1^2(e) \text{d} e \le 1 $.  
 Hence, the value of the integral could only increase if we considered functions like $\tilde{N}_i$. 

In addition, the final line is derived by bounding $P_{\tau}(e_1- e_2)$ by $|e_1 - e_2|^{-2}$ and using scaling and we obtain the result. 
 
\end{proof}

By a very similar technique, we can also prove the following estimates,
\begin{lem} \label{lem:Mdeltrest}

Assume that $M$ has support $[-S,S]^4$ and that $M$ is bounded from above  by $B$ and from below by $b$ on this support. Then we have the following estimates:
if $|z_1 -z_2| \ge   4 \sqrt{d} |S| + 2 \delta $, we have that
\begin{equation*}
\begin{aligned}
    &\sup_{\substack{N_1,N_2: \\ \int_{\mathbb{R}^4} N_i^2(e) M(e) \text{d} e =1}} \int_{(\mathbb{R}^4)^2} N_1(e_1) M(e) P_{\tau}(z_1 + e_1 - z_2 - e_2) \\
    & \hspace{2 cm} \times \mathbbm{1}(z_1 + e_1 - z_2 - e_2 \le \delta) N_2(e_2) M(e_2) \text{d}e_1 \text{d}e_2 = 0.
\end{aligned}
\end{equation*}

If instead, we had that $|z_1 - z_2| \le 4 \sqrt{d}|S| + 2 \delta$, we can instead derive following estimates, for universal constant $C$ (only depending on $M$): 
\begin{equation*} 
\begin{aligned}
    & \sup_{\substack{N_1,N_2: \\ \int_{\mathbb{R}^4} N_i^2(e) M(e) \text{d} e =1}} \int_{(\mathbb{R}^4)^2} N_1(e_1) M(e) P_{\tau}(z_1 + e_1 - z_2 - e_2) \\
    & \hspace{2 cm} \times \mathbbm{1}(z_1 + e_1 - z_2 - e_2 \le \delta) N_2(e_2) M(e_2) \text{d}e_1 \text{d}e_2 \le C \delta^2. 
\end{aligned}
\end{equation*}

\end{lem}
\begin{proof}
When $|x_1 - x_2| \ge |z_1 - z_2| + 2\delta$, the assertion is clear. We now need to consider the case that $|z_1 - z_2| \le 2 \sqrt{d}|S| + 2 \delta$. 

Consider the scaled integer lattice $\mathbb{Z}^4_{\delta}:=\frac{\delta}{\sqrt{d}} \mathbb{Z}^4 $. For each $e$ in $[-S,S]^4$, let $m$ be the closest integral point in $\mathbb{Z}^4_{\delta}$ to $e$. 
Hence, $|m_1 -e_1| \le \delta$. If $|z_1 - z_2 + e_1 - e_2 |\le \delta$, then we must have that $|m_1 - (e_2 + z_2 -z_1) | \le 2 \delta$.
Thus, we can make the following integral bound,
\begin{equation}  \label{eq:splitbycube}
\begin{aligned}
    &  \int_{(\mathbb{R}^4)^2} N_1(e_1) M(e) P_{\tau}(z_1 + e_1 - z_2 - e_2) \\
    & \hspace{ 2 cm}\mathbbm{1}(z_1 + e_1 - z_2 - e_2) \le \delta) N_2(e_2) M(e_2) \text{d}e_1 \text{d}e_2\\
    & \le \sum_{\substack{ k \in \mathbb{Z}^4_{\delta}\\ |k| \le 2 \sqrt{d}|S|} } 
    \int_{([-\frac{2\delta}{\sqrt{d}}, \frac{2\delta}{\sqrt{d}}]^4)^2} \text{d}\hat{e}_1 \text{d} \hat{e}_2 N_1(k_1 + \hat{e}_1) M(k_1 + \hat{e}_1)\\
    & \times P_{\tau}(\hat{e}_1 -\hat{e}_2) N_2(z_2 + \hat{e}_2 - z_1 - k_1)  M(z_2 + \hat{e}_2 -z_1 - k_1).
\end{aligned}
\end{equation}
The second line chooses the change of variables $e_1 \to k_1 +\hat{e}_1$ and $e_2 \to \hat{e}_2 + z_2 - z_1 - k_1$. The restriction that $|z_1 + e_1 -z_2 - e_2| \le \delta$ will ensure that any tuple $(e_1,e_2)$ that corresponds to a non-zero term in the integral on the left hand side will belong to one of the squares on the right hand side. 

We now define,
$$
\begin{aligned}
&V_1(k):=\int_{[-\frac{2\delta}{\sqrt{d}}, \frac{2\delta}{\sqrt{d}}]^4} N_1(k + \hat{e})^2 M(k_1 + \hat{e}) \text{d} \hat{e},\\
& V_2(k):= \int_{[-\frac{2\delta}{\sqrt{d}}, \frac{2\delta}{\sqrt{d}}]^4} N_2(z_2 + \hat{e} -z_1 -k )^2 M(z_2 +\hat{e} - z_1 -k)M(z_2 + \hat{e} -z_1 - k) \text{d} \hat{e}.
\end{aligned}
$$
One observation we will use is that $\sum_{k \in \mathbb{Z}^4_{\delta}} V_i(k) \le 4$ if we assume that $\int_{(\mathbb{R})^4} (N_i(e))^2 \text{d} e =1$.  
Returning to the last line of \eqref{eq:splitbycube} and using the inequality of \eqref{eq:z1z2close}, we can derive the desired bound,
\begin{equation*}
    \begin{aligned}
            &  \int_{(\mathbb{R}^4)^2} N_1(e_1) M(e) P_{\tau}(z_1 + e_1 - z_2 - e_2) \\
     & \hspace{ 2 cm} \mathbbm{1}(z_1 + e_1 - z_2 - e_2) \le \delta) N_2(e_2) M(e_2) \text{d}e_1 \text{d}e_2\\
     \le& \sum_{\substack{ k \in \mathbb{Z}^4_{\delta}\\ |k| \le 2 \sqrt{d}|S|} } \sqrt{V_1(k)} \sqrt{V_2(k)}\frac{B^2 (3 \delta)^2}{bd}  \mathfrak{I} \le \frac{B^2 (3 \delta)^2}{bd} \mathfrak{I} (\sum_{k \in \mathbb{Z}^4_{\delta}} V_1(k)) (\sum_{k \in \mathbb{Z}^4_{\delta}} V_2(k)) \le \frac{16B^2(3\delta)^2}{bd} \mathfrak{I}. 
    \end{aligned}
\end{equation*}
\end{proof}

As a corollary of this estimate, we can establish the following, 

\begin{cor} \label{cor:awayfrom0}
Let $M$ be a function of finite support that is bounded from above and away from $0$ on its support. 
Uniformly in $x_1,x_2$ and $\epsilon$, we have that,
\begin{equation*}
    O^{\delta}_{x_1,x_2,\epsilon,M} \le C(\delta),
\end{equation*}
where $C(\delta)$ is a constant that, depends on $M$ but not on $x_1,x_2$ or $\epsilon$, and goes to $0$ as $\delta$ goes to $0$.
\end{cor}

\begin{proof}
By the alternative expression found in \eqref{eq:Orewrite} and the bounds on the interior expression found in \eqref{lem:Mdeltrest}, we have the bound
\begin{equation} \label{eq:Odelx1}
\begin{aligned}
    O^{\delta}_{x_1,x_2,\epsilon,M} &\le C \delta^2\sup_{\substack{k,F_1,F_2\\
    \frac{1}{\epsilon^4} \sum_{ z \in  \epsilon \mathbb{Z}^4} k^2(z) =1\\
    \frac{1}{\epsilon^4} \sum_{z \in  \epsilon \mathbb{Z}^4} (F_i(z))^2 =1}} \frac{1}{\epsilon^{2d}} \sum_{z_1,z_2 \in \epsilon \mathbb{Z}^4} \sqrt{k}(z_1) F_1(z_1) \\
    & \hspace{4 cm} \times \mathbbm{1}(|z_1 - z_2| \le 4 \sqrt{d} S + 2\delta) \sqrt{k}(z_2) F_2(z_2). 
\end{aligned}
\end{equation}
Notice now that $G_i(z) = \sqrt{k(z)} F_i(z)$ satisfies
$$
\frac{1}{\epsilon^4} \sum_{z \in \epsilon \mathbb{Z}^4} G_i(z)^{4/3} \text{d}z \le \left[\frac{1}{\epsilon^4} \sum_{z \in \epsilon \mathbb{Z}^4} k(z)^2 \right]^{1/3} \left[\frac{1}{\epsilon^4} \sum_{z \in \epsilon \mathbb{Z}^4} (F_i)^2(z)\right]^{2/3} \le 1.
$$
We thus see that,
\begin{equation*}
    \begin{aligned}
        & \frac{1}{\epsilon^{2d}} \sum_{z_1,z_2 \in \epsilon \mathbb{Z}^4} G_1(z_1)  \mathbbm{1}(|z_1 - z_2| \le 4 \sqrt{d} S + 2\delta) G(z_2)  \\
        & \le \left[\frac{1}{\epsilon^{2d}} \sum_{z_1,z_2 \in \epsilon \mathbb{Z}^4}(G_1(z_1))^{4/3} (G_2(z_2))^{4/3} \right]^{1/2}\\
        & \times \left[\frac{1}{\epsilon^4} \sum_{y \in \epsilon \mathbb{Z}^4}\mathbbm{1}(|y| \le 4 \sqrt{d}S + 2 \delta) \frac{1}{\epsilon^4}\sum_{z_1 \in \epsilon \mathbb{Z}^4}  (G_1(z_1))^{2/3} (G_2(z_1 + y))^{2/3}\right]^{1/2}\\
        & \le \left[\frac{1}{\epsilon^4} \sum_{y \in \epsilon \mathbb{Z}^4}  \mathbbm{1}(|y| \le 4 \sqrt{d}S + 2 \delta) \left\{\frac{1}{\epsilon^4} \sum_{z_1 \in \epsilon \mathbb{Z}^4}(G_1(z_1))^{4/3} \right\}^{1/2} \left\{ \frac{1}{\epsilon^4} \sum_{z_1 \in \epsilon \mathbb{Z}^4} (G_2(z_1 + y))^{4/3} \right\}^{1/2}\right]^{1/2}\\
        & \le [4 \sqrt{d} S + 2 \delta]^{2}.
    \end{aligned}
\end{equation*}
The sequence of steps follows from carefully applying the Cauchy-Schwarz inequality twice. 
Putting the above inequality into equation \eqref{eq:Odelx1} will give us the desired bound on $O^{\delta}_{x_1,x_2,\epsilon,M}$ uniformly in $x_1,x_2$ and $\epsilon$. 
\end{proof}

At this point, we finally have enough tools to assert the following theorem.
\begin{thm} \label{thm:oanaleps}
Let $M$ be a function of finite support that is bounded from above and bounded away from $0$ on its support.

Then, uniformly in $x_1,x_2,\epsilon$ and $M$, 
\begin{equation*}
\lim_{\epsilon \to 0} O_{x_1,x_2,\epsilon,M} \le O_M.
\end{equation*}
\end{thm}
\begin{proof}
It is manifestly clear that $\lim_{\epsilon \to 0} O^{C}_{x_1,x_2,\epsilon,M} \le O^C_M$, by using the continuity of the function $C$.  Namely, one can take functions in $K_\epsilon$ and $F_{\epsilon,M}$ and generate a function in $K$, $F$ respectively by setting $k(z)= k_{\epsilon}\left(\lfloor \frac{z}{\epsilon} \rfloor \right)$ with $k$ in $K$ and $k_{\epsilon}$ in $K_{\epsilon}$ and similarly for $F$. 
For any value $\kappa$, one can find $\epsilon$ small enough so that $|C(z+ \alpha ) - C(z)| \le \kappa C(z)$ for $|\alpha| \le \epsilon$. 
Thus,
$$
\begin{aligned}
& (\epsilon)^{2d} \sum_{z_1 \in \epsilon \mathbb{Z}^4,z_2 \in \epsilon \mathbb{Z}^4}
   \int_{(\mathbb{R}^4)^2} \sqrt{k_{\epsilon}}(z_1) (f_1)_{\epsilon}(z_1,e_1) M(e_1) C(z_1 + x_1+e_1 - z_2 - x_2-e_2) M(e_2) \\
   & \times \sqrt{k_{\epsilon}}(z_2) (f_2)_{\epsilon}(z_2,e_2) \text{d}e_1 \text{d}e_2\\
    & \le [1+\kappa] \int_{(\mathbb{R}^4)^2} \text{d}z_1 \text{d}z_2\int_{(\mathbb{R}^4)^2} \sqrt{k_{\epsilon}}(z_1) (f_1)_{\epsilon}(z_1,e_1) M(e_1) C(z_1 +e_1 - z_2 -e_2) M(e_2) \\
   & \times \sqrt{k_{\epsilon}}(z_2) (f_2)_{\epsilon}(z_2,e_2) \text{d}e_1 \text{d}e_2.
\end{aligned}
$$
This is true for all $f_1,f_2$ and $k$. By choosing the supremum of $f_1$,$f_2$ and $k$ on the left hand side, we see that $\lim_{\epsilon \to 0} O^{C}_{x_1,x_2,\epsilon,M} \le [1+\kappa] O^C_{M}$. 
As $\epsilon \to 0$, the factor $\kappa$ goes to $0$. This gives the claimed statement that $\lim_{\epsilon} O^{C}_{x_1,x_2,\epsilon,M} \le O^C_M \le O_M$.

Finally, we observe that $\lim_{\epsilon \to 0}O_{x_1,x_2,\epsilon,M} \le \lim_{\epsilon \to 0} O^C_{x_1,x_2,\epsilon,M} + \lim_{\epsilon \to 0} O^{\delta}_{x_1,x_2,\epsilon,M} \le O_M + C(\delta)$. Here, we applied our earlier claim on $O^C$ along with Corollary \ref{cor:awayfrom0}. Since we can take $\delta\to 0$ after all these steps, we have the desired inequality. 
\end{proof}

\section*{Acknowledgment}
The authors would like to thank Amir Dembo for his useful suggestions. 
The authors are also grateful to Makoto Nakamura for his helpful comments.

\end{document}